\documentclass[11pt]{amsart}
\usepackage{amsmath,amsthm,amsfonts,amscd,amssymb,eucal,latexsym,mathrsfs, yhmath}
\usepackage[numbers,sort&compress]{natbib}
\usepackage[all, cmtip]{xy}

\newtheorem{theorem}{Theorem}[section]
\newtheorem{corollary}[theorem]{Corollary}
\newtheorem{lemma}[theorem]{Lemma}
\newtheorem{proposition}[theorem]{Proposition}

\theoremstyle{definition}
\newtheorem{definition}[theorem]{Definition}
\newtheorem{remark}[theorem]{Remark}

\newtheorem{question}[theorem]{Question}

\theoremstyle{plain}

\theoremstyle{definition}

\theoremstyle{remark}

\newcommand{\topol}{{\text{\rm top}}}

\newcommand{\cB}{{\mathcal B}}

\newcommand{\cJ}{{\mathcal J}}

\newcommand{\cL}{{\mathcal L}}
\newcommand{\cW}{{\mathcal W}}

\newcommand{\cC}{{\mathcal C}}

\newcommand{\cZ}{{\mathcal Z}}
\newcommand{\cV}{{\mathcal V}}

\newcommand{\cY}{{\mathcal Y}}

\newcommand{\fB}{{\mathfrak B}}
\newcommand{\fF}{{\mathfrak F}}
\newcommand{\fG}{{\mathfrak G}}
\newcommand{\fI}{{\mathfrak I}}

\newcommand{\Zb}{{\mathbb Z}}

\newcommand{\Rb}{{\mathbb R}}
\newcommand{\Nb}{{\mathbb N}}

\newcommand{\diam}{{\rm diam}}

\newcommand{\IE}{{\rm IE}}

\newcommand{\IN}{{\rm IN}}

\newcommand{\tr}{{\rm tr}}

\newcommand{\oA}{{\boldsymbol{A}}}
\newcommand{\oB}{{\boldsymbol{B}}}

\newcommand{\oU}{{\boldsymbol{U}}}

\newcommand{\ox}{{\boldsymbol{x}}}

\newcommand{\sW}{{\mathscr W}}

\newcommand{\Sym}{{\rm Sym}}

\newcommand{\Map}{{\rm Map}}
\newcommand{\eps}{\varepsilon}

\allowdisplaybreaks

\begin{document}


\title[Independence]{Combinatorial independence and sofic entropy}

\author{David Kerr}
\author{Hanfeng Li}
\address{\hskip-\parindent
David Kerr, Department of Mathematics, Texas A{\&}M University,
College Station, TX 77843-3368, U.S.A.}
\email{kerr@math.tamu.edu}

\address{\hskip-\parindent
Hanfeng Li, Department of Mathematics, Chongqing University,
Chongqing 401331, China.
Department of Mathematics, SUNY at Buffalo,
Buffalo, NY 14260-2900, U.S.A.}
\email{hfli@math.buffalo.edu}

\date{November 30, 2012}

\begin{abstract}
We undertake a local analysis of combinatorial independence as it connects to topological entropy
within the framework of actions of sofic groups.
\end{abstract}

\maketitle

\section{Introduction}

Among the various phenomena in dynamics
associated with randomness, weak mixing and entropy stand out
for the depth of their theory and the breadth of their applications (see
for example \cite{EinWar11,Gla03}).
In the setting of discrete acting groups, weak mixing makes sense in general
while entropy, as classically formulated, requires the group to be amenable.
One can view these two concepts in a unified way across
both measurable and topological dynamics by means of the combinatorial notion of independence.
In close parallel with the $\ell_1$ theorems of Rosenthal and Elton-Pajor in Banach space theory,
weak mixing and positive entropy
reflect two of the basic regimes in which combinatorial independence can occur across the
orbit of a tuple of sets in a dynamical system. The first of these asks for independence over a
subset of the group having infinite cardinality, while the other requires
this subset to satisfy a positive density condition.
Inspired by work in the local theory of entropy \cite{GlaYe09}, the authors studied
this connection between independence, weak mixing, and entropy in \cite{KerLi07,KerLi09}
as part of a program to develop a general theory of
combinatorial independence in dynamics.

Combinatorial independence is the basic set-theoretic expression of randomness in which
we are concerned not with the size of intersections, as in the probabilistic context, but
merely with their nonemptiness.
A collection $\{ (A_{i,1} , \dots , A_{i,k} ) \}_{i\in J}$
of $k$-tuples of subsets of a set $X$ is {\it independent}
if for every nonempty finite set $F\subseteq J$ and
function $\omega : F\to \{ 1,\dots,k\}$ the intersection $\bigcap_{i\in F} A_{i,\omega(i)}$ is nonempty.
If a group $G$ is acting on $X$ then given a tuple $(A_1 ,\dots,A_k )$ of subsets of $X$
we say that a set $J\subseteq G$ is an {\it independence set for $(A_1 ,\dots,A_k )$}
if the collection $\{ (s^{-1} A_1 ,\dots ,s^{-1} A_k) \}_{s\in J}$ is independent.
In the case of an action of a countable amenable group $G$ on a compact Hausdorff space $X$,
the {\it independence density} $I(\oA)$ of a tuple $\oA = (A_1 ,\dots, A_k )$ of subsets of $X$
is defined as the limit of $\varphi_{\oA}(F) /|F|$ as
the nonempty finite set $F\subseteq G$ becomes more and more left invariant,
where $\varphi_{\oA}(F)$ denotes the maximum of the cardinalities of the independence sets for $\oA$
which are contained in $F$ \cite[Prop.\ 3.23]{KerLi07}.
We then say that a tuple $(x_1 , \dots , x_k ) \in X^k$ is an {\it IE-tuple} if for every product neighbourhood
$U_1 \times\cdots\times U_k$ of $(x_1 , \dots , x_k )$ the tuple $\oU = (U_1 ,\dots, U_k )$ satisfies $I(\oU ) > 0$.
This condition on $\oU$ is equivalent to the existence
of an independence set $J\subseteq G$ for $\oU$ which has positive density with respect to
a given tempered F{\o}lner sequence $\{ F_i \}_{i\in \Nb}$ in the sense that
$\lim_{i\to\infty}  |F_i \cap J|/|F_i| > 0$.
It turns out that a nondiagonal tuple is an IE-tuple if and
only if it is an entropy tuple \cite[Sect.\ 3]{KerLi07}.
For general acting $G$ we define IT-tuples in the same way as IE-tuples except that
we ask instead for the independence set to have infinite cardinality.
Versions of IE-tuples and IT-tuples can also be defined for probability-measure-preserving actions
by requiring the condition on the independence set to hold whenever we remove small parts
of the sets $s^{-1}U_1, \dots, s^{-1}U_k$ for each $s\in G$.
The following are some of the results relating independence, weak mixing, and entropy that
were established in \cite{KerLi07,KerLi09}.
\begin{enumerate}
\item A continuous action $G\curvearrowright X$ of an Abelian group on a compact Hausdorff space
is (topologically) weakly mixing if and only if every tuple of points in $X$ is an IT-tuple
(in which case the action is said to be {\it uniformly untame of all orders}).

\item A probability-measure-preserving action $G\curvearrowright (X,\mu )$ of an
arbitrary group is weakly mixing if and only if its
universal topological model is uniformly untame of all orders.

\item A continuous action $G\curvearrowright X$ of a discrete amenable group on a compact Hausdorff space
has positive topological entropy if and only if it has a nondiagonal IE-pair
(for $G=\Zb$ this was first proved by Huang and Ye in \cite{HuaYe06} using measure-theoretic techniques).
Moreover, the action has uniformly positive entropy
if and only if every pair of points in $X$ is an IE-pair, and uniformly positive entropy of all orders
if and only if every tuple of points in $X$ is an IE-tuple.

\item A probability-measure-preserving action $G\curvearrowright (X,\mu )$ of a discrete
amenable group has positive measure entropy if and only if there is a nondiagonal
measure IE-pair in some topological model.
Moreover, the action has complete positive entropy if and only if every tuple of points is an IE-tuple
in the universal topological model  (for $G=\Zb$ this was proved
by Glasner and Weiss in \cite{GlaWei95}).
\end{enumerate}
Chung and the second author applied IE-tuples in \cite{ChuLi11} as part of a new approach
for studying the relation between homoclinicity and entropy in expansive algebraic actions
that enabled them to break the commutativity barrier and establish some duality-type equivalences
for polycyclic-by-finite acting groups. In this case, and more generally
for actions of a countable amenable group on a compact group
$X$ by automorphisms, the analysis of IE-tuples is governed by a single closed invariant normal
subgroup of $X$ called the {\it IE group} \cite[Sect.\ 7]{ChuLi11}.

Recent seminal work of Bowen in \cite{Bow10} has expanded
the scope of the classical theory of entropy for actions of discrete amenable groups
to the much broader realm of sofic acting groups.
For a countable group $G$, soficity can be expressed as the existence of a sequence
$\Sigma = \{ \sigma_i : G\to\Sym (d_i ) \}$ of maps from $G$ into finite permutation groups
which is asymptotically multiplicative and free in the sense that
\begin{enumerate}
\item[(i)] $\displaystyle\lim_{i\to\infty} \frac{1}{d_i}
\big| \{ k\in \{ 1,\dots ,d_i \} : \sigma_{i,st} (k) = \sigma_{i,s} \sigma_{i,t} (k) \} \big| = 1$
for all $s,t\in G$, and

\item[(ii)] $\displaystyle\lim_{i\to\infty} \frac{1}{d_i} \big| \{ k\in \{ 1,\dots ,d_i \} : \sigma_{i,s} (k) \neq \sigma_{i,t} (k) \} \big| = 1$
for all distinct $s,t\in G$.
\end{enumerate}
Such a sequence for which $\lim_{i\to\infty} d_i = \infty$ we call a {\it sofic approximation sequence}.
By measuring the asymptotic exponential growth of dynamical models which are compatible
with a fixed sofic approximation sequence, Bowen defined in \cite{Bow10} a collection of invariants
for probability-measure-preserving actions of a countable sofic group admitting a generating
partition with finite Shannon entropy. A remarkable application of this sofic measure entropy
was a far-reaching extension of the Ornstein-Weiss classification of Bernoulli actions of amenable groups.

The authors developed in \cite{KerLi11} a more general operator-algebraic approach to sofic entropy
that enables one to remove the generator hypothesis (see also \cite{Ker12} for a
formulation in terms of finite partitions). This led to a sofic version of topological entropy
as well as a variational principle relating it to sofic measure entropy.
We used this variational principle to compute the sofic topological
entropy of a principal algebraic action $G\curvearrowright \widehat{\Zb G/\Zb G f}$ of a
countable residually finite group in the case that
the sofic approximation sequence arises from finite quotients of $G$ and $f$ is invertible in the full group
C$^*$-algebra $C^* (G)$. In line with previous work on algebraic actions \cite{LinSchWar90,Den06,Li12,Bow11}
(see also the more recent \cite{LiTho12}),
this value turns out to be equal to the logarithm of the Fuglede-Kadison determinant of $f$ in the group
von Neumann algebra $\cL G$. We also showed how topological entropy can be used to give a
proof that Gottschalk's surjunctivity conjecture holds for countable sofic groups, a result
originally established by Gromov in \cite{Gro99}, where the idea of soficity itself first appeared.

In the present work we initiate a local analysis of independence as it connects to topological entropy
within this broadened framework of actions of sofic groups. Given a continuous action $G\curvearrowright X$
of a countable sofic group and a sofic approximation sequence $\Sigma =  \{ \sigma_i : G\to\Sym (d_i ) \}$ for $G$,
we define the notion of a $\Sigma$-IE-tuple by externalizing the positive independence density condition
in the amenable case to the finite sets $\{ 1,\dots , d_i \}$ appearing in the sequence $\Sigma$
(Definition~\ref{D-Sigma-IE}). We show in Section~\ref{S-Sigma} that
$\Sigma$-IE-tuples share many of the same properties as IE-tuples for actions of discrete
amenable groups. In particular, the action $G\curvearrowright X$ has positive entropy with respect to $\Sigma$
if and only if there is a nondiagonal $\Sigma$-IE-pair in $X\times X$. On the other hand, we do not know
whether the product formula holds in general for $\Sigma$-IE-tuples. However, granted that we use a free
ultrafilter $\fF$ over $\Nb$ to express the independence density condition in the definition of $\Sigma$-IE-tuples,
we demonstrate in Theorem~\ref{T-ergodic to product} that the product formula holds
under the assumption of ergodicity on the action of the commutant of $G$
inside the group of measure-preserving automorphisms of the Loeb space $\prod_\fF \{ 1,\dots ,d_i \}$ which arise
from permutations of the sets $\{ 1,\dots ,d_i \}$. We then prove that this commutant acts ergodically
when $G$ is residually finite and $\Sigma$ is built from finite quotients of $G$ (Theorem~\ref{T-rf ergodic}), and also
when $G$ is amenable and $\Sigma$ is arbitrary (Theorem~\ref{T-amenable ergodic}).
In the case that  $G$ is nonamenable, a combination of results of Elek and Szabo  \cite[Thm.\ 2]{EleSza11} and
Paunescu \cite{Pau11} shows that there exist $\Sigma$ for which the ergodicity condition fails.

The definition of IE-tuples for amenable $G$, as given in \cite{KerLi07}, involves an asymptotic
density condition over finite subsets of $G$ which become more and more invariant.
Although density in this sense loses its meaning in the nonamenable case,
we might nevertheless ask what the external independence density
in the definition of $\Sigma$-IE-tuples implies about the degree of independent behaviour across
orbits in $X$. We observe in Proposition~\ref{P-sofic IE to orbit IE}
that every $\Sigma$-IE-tuple (and more generally every
sofic IE-tuple as defined in Definition~\ref{D-Sigma-IE}) is an {\it orbit IE-tuple}, by which we mean
that for every product neighbourhood $U_1 \times\cdots\times U_k$ of the given tuple
$(x_1,\dots,x_k)$ in $X^k$, the tuple $(U_1 , \dots , U_k )$
has positive independence density over $G$ in the sense that
there is a $q>0$ such that every finite set $F\subseteq G$ has a subset of cardinality at least $q|F|$
which is an independence set for $(U_1 , \dots , U_k )$ (note that this definition makes sense
for any acting group $G$). We show moreover in Theorem~\ref{T-amenable case: orbit IE=IE} that, for amenable $G$,
$\Sigma$-IE-tuples, IE-tuples, and orbit IE-tuples are all the same thing. This puts us in the pleasant
and somewhat surprising situation that IE-tuples can be identified by a density condition that does
not structurally depend on amenability for its formulation,
and raises the question about the relation between entropy and orbit IE-tuples for nonamenable sofic $G$.
In another direction, Theorem~\ref{T-orbit IE to nontame} asserts that if a tuple of subsets of $X$
has positive independence density over $G$ then it has an infinite independence set in $G$,
which implies that every orbit IE-tuple is an IT-tuple.

By a theorem of Chung and the second author in \cite{ChuLi11}, an algebraic action of
a countable group $G$ is expansive if and only if it is either the dual action
$G\curvearrowright X_A := \widehat{(\Zb G)^n /(\Zb G)^n A}$
for some $n\in\Nb$ and matrix $A\in M_n (\Zb G)$ which is invertible in $M_n (\ell^1 (G))$,
or the restriction of such an action to a closed $G$-invariant subgroup of $X_A$.
In the same paper it is shown that an expansive algebraic action $G\curvearrowright X$
of a polycyclic-by-finite group has completely positive entropy with respect to the Haar measure
prescisely when the IE group is equal to $X$, which is equivalent to every tuple of points
in $X$ being an IE-tuple. It is also shown that, when $G$ is amenable, every action of the form
$G\curvearrowright X_A$ with $A$ invertible in $M_n (\ell^1 (G))$ has the property that every
tuple of points in $X$ is an IE-tuple (see Lemma~5.4 and Theorems~7.3 and 7.8 in \cite{ChuLi11}).
We prove in Theorem~\ref{T-invertible to UPE} that if $G$ is a countable
sofic group, $n\in\Nb$, and $A$ is a matrix in $M_n (\Zb G)$ which is invertible in $M_n (\ell^1 (G))$,
then the algebraic action $G\curvearrowright X_A$ has the property that every tuple of points
in $X_A$ is a $\Sigma$-IE-tuple for every sofic approximation sequence $\Sigma$.
We use this to answer a question of Deninger in the case that $G$ is residually finite by
combining it with an argument from \cite{ChuLi11} and the entropy computation for
principal algebraic actions from \cite{KerLi11} mentioned above to
deduce that if $f$ is an element of $\Zb G$ which is invertible in $\ell^1 (G)$
and has no left inverse in $\Zb G$ then the Fuglede-Kadison determinant of $f$ satisfies
$\det_{\cL G} f > 1$ (Corollary~\ref{C-answer to Deninger}). Deninger asked whether this holds for all countable
groups \cite[Question 26]{Den09}, and affirmative answers were given in \cite{DenSch07} for residually finite amenable $G$
and more generally in \cite{ChuLi11} for amenable $G$.

For a continuous action $G\curvearrowright X$ of a countably infinite group on a compact metrizable
space with compatible metric $\rho$, we say that a pair $(x, y)\in X\times X$ is a {\it Li-Yorke pair} if
\[
\limsup_{G \ni s\to \infty}\rho(sx, sy)>0 \hspace*{5mm}\text{and} \hspace*{5mm}
\liminf_{G \ni s\to \infty}\rho(sx, sy)=0 ,
\]
where the limit supremum and limit infimum
mean the limits of $\sup_{s\in G\setminus F} \rho(sx, sy)$ and\linebreak $\inf_{s\in G\setminus F} \rho(sx, sy)$, respectively,
over the net of finite subsets $F$ of $G$.
Note that the definition of Li-Yorke pair does not depend on the choice of the metric $\rho$.
The action $G\curvearrowright X$ is said to be {\it Li-Yorke chaotic} if there is an uncountable subset
$Z$ of $X$ such that every nondiagonal pair $(x, y)$ in $Z\times Z$ is a Li-Yorke pair.
The notion of Li-Yorke chaos stems from \cite{LiYor75}.
In the case of a continuous map $T:X\to X$, a theorem
Blanchard, Glasner, Kolyada, and Maass in \cite{BlaGlaKolMaa02} states that
positive entropy implies Li-Yorke chaos. In \cite{KerLi07} the authors strengthened
this by showing that for every $k\geq 2$ and product neighbourhood $U_1\times\cdots\times U_k$
of a nondiagonal IE-tuple $(x_1,\dots,x_k)\in X^k$ there are Cantor sets $Z_i \subseteq U_i$
for $i=1,\dots,k$ such that
\begin{enumerate}
\item every nonempty tuple of points in $\bigcup_i Z_i$ is an IE-tuple, and

\item for all $m\in\Nb$, distinct $y_1, \dots ,y_m\in \bigcup_i Z_i$, and $y_1' ,\dots , y_m'\in \bigcup_i Z_i$
one has
\[
\liminf_{n\to\infty} \max_{1\leq i\leq m} \rho (T^n y_i,y_i' ) = 0 .
\]
\end{enumerate}
In Theorem~\ref{T-positive entropy to chaos} we show that a similar result holds when
$G$ is sofic and IE-tuples are replaced by $\Sigma$-IE-tuples as defined with respect to a
free ultrafilter $\fF$ on $\Nb$, where $\Sigma$ is any sofic approximation sequence for $G$.
Using $\fF$ in the definition of entropy, we deduce that if the action has positive entropy
for some $\Sigma$ then it is Li-Yorke chaotic. As a corollary, if the action $G\curvearrowright X$ is distal then
$h_\Sigma (X,G)=0$ or $-\infty$. In particular, when $G$ is amenable every distal action $G\curvearrowright X$ has zero entropy,
which is well known in the case $G=\Zb$ \cite{Parry}.

The following diagram illustrates how some of the main results of the paper
relate various properties of actions
of a countable discrete group $G$ on a compact metrizable space $X$,
which we assume to have more than one point.
In the left column we assume that $G$ is sofic and that $\Sigma$ is a fixed but
arbitrary sofic approximation sequence. The unlabeled implications are trivial.
By pair we mean an element of $X\times X$.
See \cite{KerLi07} for terminology related to tameness and nullness.

\vspace*{1.5mm}
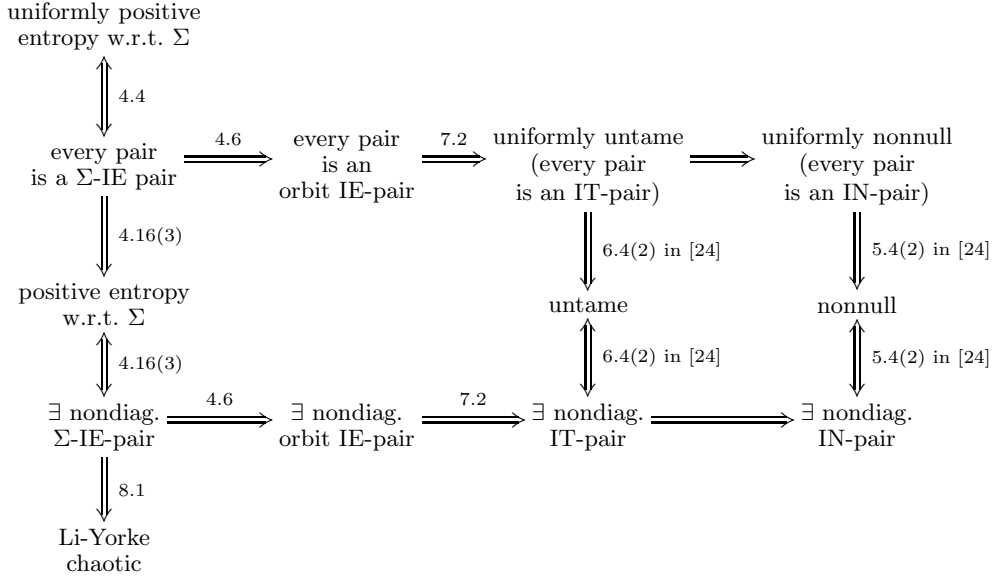
\[
\footnotesize
\xymatrix{
\txt{uniformly positive\\ entropy w.r.t.\ $\Sigma$} \ar@{<=>}[d]^-*+{\text{\tiny\ref{R-entropy tuple}}} &&& \\
\txt{every pair\\ is a $\Sigma$-IE pair} \ar@=[d]^-*+{\text{\tiny\ref{P-basic}(3)}}\ar@=[r]^-*+{\text{\tiny\ref{P-sofic IE to orbit IE}}} &
\txt{every pair\\ is an\\ orbit IE-pair}\ar@=[r]^-*+{\text{\tiny\ref{C-orbit IE to IT}}} &
\txt{uniformly untame\\ (every pair\\ is an IT-pair)} \ar@=[d]^-*+{\text{\tiny 6.4(2) in \cite{KerLi07}}} \ar@=[r] &
\txt{uniformly nonnull\\ (every pair\\ is an IN-pair)} \ar@=[d]^-*+{\text{\tiny 5.4(2) in \cite{KerLi07}}}  \\
\txt{positive entropy\\ w.r.t.\ $\Sigma$} \ar@{<=>}[d]^-*+{\text{\tiny\ref{P-basic}(3)}}&&
\txt{untame} \ar@{<=>}[d]^-*+{\text{\tiny 6.4(2) in \cite{KerLi07}}} &
\txt{nonnull} \ar@{<=>}[d]^-*+{\text{\tiny 5.4(2) in \cite{KerLi07}}} &\\
\txt{$\exists$ nondiag.\\ $\Sigma$-IE-pair} \ar@=[d]^-*+{\text{\tiny\ref{T-positive entropy to chaos}}} \ar@=[r]^-*+{\text{\tiny\ref{P-sofic IE to orbit IE}}}&
\txt{$\exists$ nondiag.\\ orbit IE-pair}\ar@=[r]^-*+{\text{\tiny\ref{C-orbit IE to IT}}} &
\txt{$\exists$ nondiag.\\ IT-pair} \ar@=[r] &
\txt{$\exists$ nondiag.\\ IN-pair} \\
\txt{Li-Yorke\\ chaotic} &&&
}
\]
\vspace*{3.5mm}

The organization of the paper is as follows.
In Section~\ref{S-entropy} we set up some basic notation and
review sofic topological entropy. In Section~\ref{S-orbit} we introduce
orbit IE-tuples and prove a product formula for them.
Section~\ref{S-Sigma} introduces $\Sigma$-IE-tuples and includes our results relating
them to orbit IE-tuples.
In Section~\ref{S-product} we focus on the product formula for $\Sigma$-IE-tuples
and the question of ergodicity for the action of $G'$ on the Loeb space.
Section~\ref{S-algebraic} contains the material on algebraic actions.
In Section~\ref{S-untame} we prove that positive independence density for a tuple
of subsets implies the existence of an infinite independence set, showing that
orbit IE-tuples are IT-tuples.
Finally in Section~\ref{S-chaos} we establish the theorem connecting
independence and entropy to Li-Yorke chaos in the sofic framework.
\medskip

\noindent{\it Acknowledgements.}
The first author was partially supported by NSF grants DMS-0900938 and DMS-1162309.
The second author was partially supported by NSF grant DMS-1001625. We are grateful to
Wen Huang, Liviu Paunescu and Xiangdong Ye for helpful comments.

\section{Sofic topological entropy}\label{S-entropy}

We review here the definition of sofic topological entropy \cite{KerLi11,KerLi10}
and in the process introduce some of the basic notation
and terminology appearing throughout the paper. Our approach will bypass
the operator algebra technology that appears in \cite{KerLi11,KerLi10}.

Let $Y$ be a set equipped with a pseudometric $\rho$ and let $\eps\geq 0$. A set $A\subseteq Y$ is said
to be {\it $(\rho ,\eps )$-separated} if $\rho (x,y) \geq \eps$ for all distinct $x,y\in A$.
Write $N_\eps (Y, \rho )$ for the maximum cardinality of a $(\rho ,\eps )$-separated subset of $Y$.

Let $G\curvearrowright X$ be a continuous action of a countable sofic group on a compact metrizable space.
Let $\Sigma = \{ \sigma_i : G\to\Sym (d_i) \}$ be a sofic approximation sequence for $G$,
meaning that
\begin{enumerate}
\item[(i)] $\displaystyle\lim_{i\to\infty} \frac{1}{d_i}
\big| \{ k\in \{ 1,\dots ,d_i \} : \sigma_{i,st} (k) = \sigma_{i,s} \sigma_{i,t} (k) \} \big| = 1$
for all $s,t\in G$,

\item[(ii)] $\displaystyle\lim_{i\to\infty} \frac{1}{d_i} \big| \{ k\in \{ 1,\dots ,d_i \} : \sigma_{i,s} (k) \neq \sigma_{i,t} (k) \} \big| = 1$
for all distinct $s,t\in G$,
\end{enumerate}
and $d_i \to\infty$ as $i\to\infty$. Depending on the situation, for $a\in\{1,\dots ,d_i \}$ we may write $\sigma_{i,s} (a)$,
$\sigma_i (s)a$, or $sa$ to denote the image of $a$ under the evaluation of $\sigma_i$ at $s$.

Let $\rho$ be a continuous pseudometric on $X$.
For a given $d\in\Nb$, we define on the set of all maps from $\{ 1,\dots ,d\}$ to $X$ the pseudometrics
\begin{align*}
\rho_2 (\varphi , \psi ) &= \bigg( \frac{1}{d} \sum_{a=1}^d (\rho (\varphi (a),\psi (a)))^2 \bigg)^{1/2} , \\
\rho_\infty (\varphi ,\psi ) &= \max_{a=1,\dots ,d} \rho (\varphi (a),\psi (a)) .
\end{align*}

\begin{definition}\label{D-map top}
Let $F$ be a nonempty finite subset of $G$ and $\delta > 0$.
Let $\sigma$ be a map from $G$ to $\Sym (d)$ for some $d\in\Nb$.
Define $\Map (\rho ,F,\delta ,\sigma )$ to be the set of all maps $\varphi : \{ 1,\dots ,d\} \to X$ such that
$\rho_2 (\varphi\sigma_s , \alpha_s \varphi ) < \delta$ for all $s\in F$, where $\alpha_s$ is the transformation
$x\mapsto sx$ of $X$.
\end{definition}

\begin{definition}
Let $F$ be a nonempty finite subset of $G$ and $\delta > 0$. For $\eps > 0$ define
\begin{align*}
h_{\Sigma ,2}^\eps (\rho ,F, \delta ) &=
\limsup_{i\to\infty} \frac{1}{d_i} \log N_\eps (\Map (\rho ,F,\delta ,\sigma_i ),\rho_2 ) ,\\
h_{\Sigma ,2}^\eps (\rho ,F) &= \inf_{\delta > 0} h_{\Sigma ,2}^\eps (\rho ,F,\delta ) ,\\
h_{\Sigma ,2}^\eps (\rho ) &= \inf_{F} h_{\Sigma ,2}^\eps (\rho ,F) ,\\
h_{\Sigma ,2} (\rho ) &= \sup_{\eps > 0} h_{\Sigma ,2}^\eps (\rho ) ,
\end{align*}
where $F$ in the third line ranges over the nonempty finite subsets of $G$. In the case that
$\Map (\rho ,F,\delta ,\sigma_i )$ is empty for all sufficiently large $i$, we set
$h_{\Sigma ,2}^\eps (\rho ,F, \delta ) = -\infty$.
We similarly define $h_{\Sigma ,\infty}^\eps (\rho ,F, \delta )$, $h_{\Sigma ,\infty}^\eps (\rho ,F)$,
$h_{\Sigma ,\infty}^\eps (\rho )$ and $h_{\Sigma ,\infty} (\rho )$
using $N_\eps(\cdot, \rho_\infty)$ in place of $N_\eps(\cdot, \rho_2)$.
\end{definition}

Instead of the limit supremum above we could have taken a limit over a fixed
free ultrafilter on $\Nb$, whose utility is apparent for example if we wish to have
a product formula (see Section~\ref{S-product}).
We will also use this variant in Section~ \ref{S-chaos}.

The pseudometric $\rho$ is said to be {\it dynamically generating} if
for every pair of distinct points $x,y\in X$ there is an $s\in G$ such that $\rho(sx, sy)>0$.

\begin{lemma} \label{L-change pseudometric}
Suppose that $\rho$ and $\rho'$ are continuous pseudometrics on $X$
and that $\rho'$ is dynamically generating.
Let $F$ be a nonempty finite subset of $G$ and $\delta>0$. Then there exist a nonempty finite subset
$F'$ of $G$ and $\delta'>0$ such that for any $d\in \Nb$ and sufficiently good sofic approximation
$\sigma: G\to \Sym(d)$ one has $\Map(\rho', F', \delta', \sigma)\subseteq \Map(\rho, F, \delta, \sigma)$.
\end{lemma}

\begin{proof}
List the elements of $G$ as $s_1, s_2, \dots$. Since $\rho'$ is dynamically generating, we have the
compatible metric $\rho''$ on $X$ defined by
\[
\rho''(x, y)=\sum_{k=1}^{\infty}\frac{1}{2^k}\rho'(s_kx, s_ky).
\]
It follows that there are a nonempty finite subset $F''$ of $G$ and a $\delta''>0$ such that, for all $x,y\in X$,
if $\max_{s\in F''}\rho'(sx, sy)<\delta''$ then $\rho(x, y)<\delta/2$.
Set $F'=F''\cup (F''F)$. Let $\delta'>0$ and let $\sigma$ be a map from $G$ to $\Sym(d)$
for some $d\in \Nb$. Let $\varphi\in \Map(\rho', F', \delta', \sigma)$. Then
\begin{align*}
|\{a\in \{1, \dots, d\}: \rho'(s_1s_2\varphi(a), &\varphi((s_1s_2)a))<\sqrt{\delta'} \text{ and}\\
\rho'(s_1 \varphi &(s_2a), \varphi(s_1(s_2a)))<\sqrt{\delta'}
\text{ for all } s_1\in F'', s_2\in F\}|\\
&\ge (1-2|F''| |F|\delta')d.
\end{align*}
Suppose that $2\sqrt{\delta'}<\delta''$ and $\sigma$ is a good enough sofic approximation for $G$ so that
\[
|\{a\in \{1, \dots, d\}: (s_1s_2)a=s_1(s_2a) \text{ for all } s_1\in F'', s_2\in F\}|\ge (1-\delta')d.
\]
Then
\begin{align*}
\lefteqn{|\{a\in \{1, \dots, d\}: \rho(s\varphi(a), \varphi(sa))<\delta/2 \text{ for all } s\in F\}|}\hspace*{15mm} \\
\hspace*{10mm} &\ge |\{a\in \{1, \dots, d\}: \rho'(s_1s_2\varphi(a), s_1\varphi(s_2a))<2\sqrt{\delta'} \text{ for all } s_1\in F'', s_2\in F\}|\\
&\ge (1-(1+2|F''| |F|)\delta')d.
\end{align*}
It follows that when $\delta'$ is small enough independently of $d$ and $\sigma$, one has $\varphi\in \Map(\rho, F, \delta, \sigma)$.
\end{proof}

The following proposition is contained in Proposition~2.4 of \cite{KerLi10}, whose
statement and proof use the operator-algebraic formulation of sofic topological entropy from \cite{KerLi11}.

\begin{proposition}\label{P-pseudometric}
Let $\rho$ and $\rho'$ be continuous pseudometrics on $X$ which are dynamically generating. Then
\[
h_{\Sigma ,2} (\rho ) = h_{\Sigma ,2} (\rho') = h_{\Sigma ,\infty} (\rho ) = h_{\Sigma ,\infty} (\rho') .
\]
\end{proposition}

\begin{proof}
Since the pseudometric $\rho_\infty$ dominates the pseudometric $\rho_2$,
we have $h_{\Sigma ,2} (\rho ) \leq h_{\Sigma ,\infty} (\rho )$.

Next we argue that $h_{\Sigma ,\infty} (\rho ) \leq h_{\Sigma ,2} (\rho )$.
Let $F$ be a finite subset of $G$, $\delta > 0$, and $\sigma$ a map from $G$ to $\Sym (d)$
for some $d\in\Nb$. Let $1/2>\eps > 0$.
Let $\eta > 0$ be the minimum of $\eps^2$ and the reciprocal of the minimum cardinality
of an $(\eps /2)$-spanning subset of $X$ with respect to $\rho$.
Given a $\varphi\in\Map (\rho,F,\delta,\sigma)$,
every element in the open $(\rho_2 ,\eta )$-ball in $\Map (\rho,F,\delta,\sigma)$ around $\varphi$
agrees with $\varphi$ to within
$\sqrt{\eta}$, and hence to within $\eps$, on a subset of $\{1,\dots,d\}$ of cardinality at least $(1-\eta )d$.
Thus the maximum cardinality of a $(\rho_\infty ,\eps )$-separated subset of the open
$(\rho_2 ,\eta)$-ball around $\varphi$ is at most $\sum_{j=0}^{\lfloor\eta d\rfloor} \binom{d}{j} \eta^{-j}$,
which by Stirling's approximation is bounded above, for all $d$ sufficiently large,
by $e^{\beta d} \eta^{-\eta d}$ for some $\beta > 0$ not depending on $d$
with $\beta\to 0$ as $\eps\to 0$. Hence
\[
N_\eps (\Map (\rho,F,\delta,\sigma),\rho_\infty )
\leq e^{\beta d} \eta^{-\eta d} N_\eta (\Map (\rho,F,\delta,\sigma),\rho_2 ) .
\]
It follows that
\[
h_{\Sigma,\infty}^\eps (\rho) \leq h_{\Sigma,2}^\eta (\rho) + \beta - \eta \log \eta ,
\]
and since $\beta - \eta \log \eta \to 0$ as $\eps\to 0$ we conclude that
$h_{\Sigma ,\infty} (\rho ) \leq h_{\Sigma ,2} (\rho )$.

Finally we show that $h_{\Sigma ,2} (\rho ) \leq h_{\Sigma ,2} (\rho' )$,
which will establish the proposition as we can interchange the roles of $\rho$ and $\rho'$.

Let $\eps > 0$.
Since $\rho$ is dynamically generating, we can find a finite set $K\subseteq G$
and an $\eps' > 0$ such that, for all $x,y\in X$, if $\rho (sx,sy) < \sqrt{3\eps'}$ for all $s\in K$
then $\rho' (x,y) < \eps /\sqrt{2}$.
By shrinking $\eps'$ if necessary we may assume that $3\eps' |K| < \eps^2 /2$.
Take a finite set $F\subseteq G$ containing $K$ and a $\delta > 0$ with $\delta\leq\eps'$
such that $h_{\Sigma,2}^{\eps'} (\rho ,F,\delta ) \leq h_{\Sigma,2}^{\eps'} (\rho ) + \eps$.
Since $\rho'$ is dynamically generating, by Lemma~\ref{L-change pseudometric} there are
a nonempty finite set $F' \subseteq G$ and a $\delta'>0$ such that for any $d\in \Nb$ and
sufficiently good sofic approximation $\sigma: G\to \Sym(d)$ we have
$\Map(\rho', F', \delta', \sigma)\subseteq \Map(\rho, F, \delta, \sigma)$.

Given $\varphi ,\psi\in\Map (\rho' ,F',\delta' ,\sigma )$ such that
$\rho_2 (\varphi , \psi ) < \eps'$, for each $s\in K$ we have, writing $\alpha_s$ for
the transformation $x\mapsto sx$ of $X$,
\begin{align*}
\rho_2 (\alpha_s \varphi , \alpha_s \psi ) \leq \rho_2 (\alpha_s \varphi , \varphi\sigma_s )
+ \rho_2 (\varphi\sigma_s , \psi\sigma_s ) + \rho_2 (\psi\sigma_s , \alpha_s \psi )
< \delta + \eps' + \delta \leq 3\eps' .
\end{align*}
This implies that there is a set
$W\subseteq \{1,\dots,d \}$ of cardinality at least $(1-3\eps' |K|)d$ such that for all $a\in W$ we have
$\rho (s\varphi (a),s\psi (a)) < \sqrt{3\eps'}$ for every $s\in K$ and hence $\rho' (\varphi (a),\psi (a)) < \eps /\sqrt{2}$.
As a consequence, assuming (as we may by normalizing) that $X$ has $\rho'$-diameter at most one,
\[
\rho'_2 (\varphi , \psi ) \leq \sqrt{(\eps/\sqrt{2})^2 + 3\eps' |K|} < \eps .
\]
It follows that
\[
N_\eps (\Map (\rho' ,F',\delta' ,\sigma ),\rho'_2 ) \leq N_{\eps'} (\Map (\rho ,F,\delta ,\sigma ),\rho_2 )
\]
and hence $h_{\Sigma,2}^\eps (\rho' ,F',\delta' ) \leq h_{\Sigma,2}^{\eps'} (\rho ,F,\delta )$, so that
\begin{align*}
h_{\Sigma ,2}^\eps (\rho')
\leq h_{\Sigma,2}^\eps (\rho' ,F',\delta' )
\leq h_{\Sigma,2}^{\eps'} (\rho ,F,\delta )
\leq h_{\Sigma,2}^{\eps'} (\rho) + \eps
\leq h_{\Sigma,2} (\rho) + \eps .
\end{align*}
Since $\eps$ was an arbitrary positive number we conclude that $h_{\Sigma ,2} (\rho' ) \leq h_{\Sigma,2} (\rho)$.
\end{proof}

\begin{definition}\label{D-topological entropy}
The {\it topological entropy} $h_\Sigma (X,G)$ of the action $G\curvearrowright X$ with respect to $\Sigma$
is defined to be the common value in Proposition~\ref{P-pseudometric} over all dynamically generating
continuous pseudometrics on $X$.
\end{definition}

Note that the approximate multiplicativity of a sofic approximation was only needed in the proof of
Lemma~\ref{L-change pseudometric} to handle the situation
in which one of $\rho$ and $\rho'$ is not an actual metric.
Indeed we could have defined topological entropy more easily by using the obvious fact that
$h_{\Sigma ,\infty} (\rho )$ takes a common value over all compatible metrics on $X$,
with Proposition~\ref{P-pseudometric} then being regarded as a Kolmogorov-Sinai theorem.
As with the $(n,\eps)$-separated set definition of topological entropy for
single transformations,
it is by considering pseudometrics that we can compute the entropy for a nontrivial
example like the shift action $G\curvearrowright \{ 1, \dots ,k \}^G$. In this case one
can see that the value is $\log k$ independently of $\Sigma$
by considering the pseudometric $\rho$ on $\{ 1, \dots ,k \}^G$
given by $\rho (x,y) = 0$ or $1$ depending on whether or not the coordinates of $x$ and $y$ at $e$ agree.
Indeed $\log k$ is easily seen to be an upper bound, and given a nonempty finite set $F\subseteq G$,
a $\delta > 0$, and a good enough sofic approximation $\sigma : G\to\Sym (d)$ we can construct
a $(\rho_\infty ,1/2)$-separated subset of $\Map (\rho ,F,\delta , \sigma )$ of cardinality $k^d$
by associating to every $\omega\in \{ 1,\dots ,k \}^d$
some $\varphi_\omega \in\Map (\rho ,F,\delta , \sigma )$ defined by
$\varphi_\omega (a)(s^{-1}) = \omega (\sigma_s (a))$ for all $a\in \{ 1,\dots , d \}$ and $s\in G$.

For actions of amenable $G$, the entropy $h_\Sigma (X,G)$ coincides with the classical
topological entropy for every $\Sigma$ \cite{KerLi11}. Such an action
always has a largest zero-entropy factor
(i.e., a zero-entropy factor such that every zero-entropy factor factors through it),
called the {\it topological Pinsker factor} \cite{BlaLac93}.
More generally for sofic $G$, with respect to a fixed $\Sigma$ there exists a largest factor of
the action $G\curvearrowright X$ which has entropy either $0$ or $-\infty$ (note that the value $-\infty$
does not occur for actions of amenable $G$).
This follows from the fact that the property of having entropy $0$ or $-\infty$
is preserved under taking countable products and restricting to closed invariant sets. (The property of having entropy $-\infty$ is also preserved under taking countable products, though we do not know what happens to the property of having entropy $0$.)
We say that the action has
{\it completely positive entropy with respect to $\Sigma$} if each of its nontrivial factors
has positive entropy with respect to $\Sigma$.

Unlike in the amenable case, the largest factor with entropy $0$ or $-\infty$ might
have factors with positive entropy. In fact for every nonamenable $G$ there exist zero-entropy
actions of $G$ which have factors with positive entropy:
Take an action $G\curvearrowright X$ with $h_\Sigma (X,G) > 0$ and an action
$G\curvearrowright Y$ which has no $G$-invariant Borel probability measure, and consider the action
of $G$ on $K:=(X\times Y) \coprod \{ z \}$ where $z$ is a point on which $G$ acts trivially.
Then $h_\Sigma (K,G) = 0$ but the quotient action on $X\coprod \{ z \}$ satisfies $h_\Sigma (X\coprod \{ z \} ,G) > 0$.

\section{Orbit IE-tuples}\label{S-orbit}

Let $G\curvearrowright X$ be a continuous action of a discrete group on a compact Hausdorff space.
Recall from the introduction that if $\oA = (A_1 , \dots ,A_k )$ is a tuple of subsets of $X$
then we say that a subset $F$ of $G$ is an {\it independence set for $\oA$} if for every
finite subset $J$ of $F$ and every function $\omega : J\to \{ 1,\dots ,k \}$
we have $\bigcap_{s\in J} s^{-1} A_{\omega (s)} \neq\emptyset$.

\begin{definition}\label{D-independence density}
Let $\oA = (A_1 , \dots ,A_k )$ be a tuple of subsets of $X$. We define the
{\it independence density of $\oA$ (over $G$)} to be the largest $q\ge 0$ such that
every finite set $F\subseteq G$ has a subset of cardinality at least $q|F|$
which is an independence set for $\oA$.
\end{definition}

\begin{definition}\label{D-orbit IE}
We say that a tuple $\ox = (x_1 ,\dots ,x_k )\in X^k$ is an {\it orbit IE-tuple} (or {\it orbit IE-pair}
in the case $k=2$) if for every product neighbourhood $U_1 \times\dots\times U_k$ of $\ox$ the tuple
$(U_1 , \dots , U_k )$ has positive independence density. Write $\IE_k (X,G)$ for the
set of all orbit IE-tuples of length $k$.
\end{definition}

As Theorem~\ref{T-amenable case: orbit IE=IE} below demonstrates,
the notation $\IE_k (X,G)$ is consistent with its use in \cite{KerLi07} when $G$ is amenable.

The equality in the next theorem statement is understood with respect to the identification of
$((x_1,\dots,x_k),(y_1,\dots,y_k))\in X^k \times Y^k$ and $((x_1,y_1),\dots,(x_k,y_k))\in (X\times Y)^k$.

\begin{theorem} \label{T-product for orbit IE}
Let $G\curvearrowright X$ and $G\curvearrowright Y$ be continuous actions on compact Hausdorff spaces. Let $k\in \Nb$. Then
\[
\IE_k (X\times Y,G)= \IE_k (X,G) \times \IE_k (Y,G).
\]
\end{theorem}

\begin{proof}
The inclusion $\IE_k (X\times Y,G)\subseteq \IE_k (X,G) \times \IE_k (Y,G)$ is trivial.
To prove the other direction, it suffices to show that if $\oA = (A_1 , \dots ,A_k )$
is a  tuple of subsets of $X$ with independence density $q$
and $\oB=(B_1, \dots, B_k)$ is a tuple of subsets of
$Y$ with independence density $r$, then
$\oA\times \oB:=(A_1\times B_1, \dots, A_k\times B_k)$ has independence density at least $qr$.
Let $F$ be a nonempty finite subset of $G$. Then we can find a $J\subseteq F$ with $|J|\ge q|F|$
which is an independence set for $\oA$. We can then find a $J_1\subseteq J$ with $|J_1|\ge r|J|$
which as an independence set for $\oB$. Then $J_1$ is an independence
set for $\oA\times \oB$ and $|J_1|\ge qr|F|$.
\end{proof}

In \cite{KerLi07} we defined a tuple $\ox = (x_1 , \dots , x_k )\in X^k$ to be an {\it IN-tuple}
if for every product neighbourhood $U_1 \times\dots\times U_k$ of $\ox$ the tuple
$(U_1 ,\dots, U_k )$ has arbitrarily large finite independence sets.
The following fact is obvious.

\begin{proposition}\label{P-orbit IE to IN}
Suppose that $G$ is infinite. Then every orbit IE-tuple is an IN-tuple.
\end{proposition}

We will strengthen this assertion in Theorem~\ref{T-orbit IE to nontame}.

\section{$\Sigma$-IE-tuples}\label{S-Sigma}

Unless otherwise stated, throughout this section $G$ is a countable sofic discrete group,
subject to further hypotheses as appropriate.
We suppose $G$ to be acting continuously on a compact metrizable space $X$,
and $\rho$ denotes a dynamically generating continuous pseudometric on $X$ unless
otherwise stated.

In order to be able to define the notion of a sofic IE-tuple as appears in Proposition~\ref{P-sofic IE to orbit IE},
we will set up our definitions for a general sofic approximation net
$\Sigma = \{ \sigma_i : G\to\Sym (d_i ) \}$, which is formally defined in the same way as
the sequential version.

\begin{definition}
Let $\oA = (A_1 , \dots ,A_k )$ be a tuple of subsets of $X$.
Let $F$ be a nonempty finite subset of $G$ and $\delta > 0$.
Let $\sigma$ be a map from $G$ to $\Sym (d)$ for some $d\in\Nb$.
We say that a set $\cJ\subseteq \{ 1,\dots ,d\}$ is a {\it $(\rho ,F,\delta ,\sigma )$-independence set for $\oA$} if for
every function $\omega : \cJ\to \{ 1,\dots ,k \}$
there exists a $\varphi \in\Map (\rho ,F,\delta ,\sigma )$ such that
$\varphi (a) \in A_{\omega (a)}$ for every $a\in \cJ$.
\end{definition}

\begin{definition} \label{D-positive independence density}
Let $\oA = (A_1 , \dots ,A_k )$ be a tuple of subsets of $X$.
Let $\Sigma = \{ \sigma_i : G\to\Sym (d_i ) \}$ be a sofic approximation net for $G$.
We say that $\oA$ has {\it positive upper independence density over $\Sigma$} if there exists a $q > 0$
such that for every nonempty finite set $F\subseteq G$ and $\delta > 0$ there is a cofinal set of $i$
for which $\oA$ has a $(\rho ,F,\delta ,\sigma_i )$-independence set of cardinality at least $qd_i$.
By Lemma~\ref{L-change pseudometric} this definition does not depend on the choice of $\rho$.
\end{definition}

For the purposes of Sections~\ref{S-product} and \ref{S-chaos},
we will consider a variation of the above definition
in which cofinality is replaced by the stronger requirement of
membership in a fixed free ultrafilter $\fF$ on $\Nb$.
The resulting notion of positive upper independence density over $\Sigma$ with respect to $\fF$
will then be used when interpreting the following definition of $\Sigma$-IE-tuples.

By the {\it universal sofic approximation net} for $G$ we mean the net $(\sigma,F)\mapsto \sigma$
indexed by the directed set of pairs $(\sigma,F)$ where $\sigma$ is a map from $G$ to $\Sym(d)$ for some $d\in\Nb$
and $F$ is a finite subset of $G$, and $(\sigma' :G\to\Sym (d'),F' )\succ (\sigma:G\to\Sym (d),F)$ means
that $d' \geq d$ and $|V(\sigma' ,F)|/d' \geq |V(\sigma ,F)|/d$, where $V(\omega,F)$ for a map
$\omega : G\to\Sym (c)$ denotes the set of all $a\in \{1,\dots ,c\}$ such that $\sigma (s)\sigma(t)a = \sigma (st)a$
for all $s,t\in F$ and $\sigma (s)a \neq \sigma (t)a$ for all distinct $s,t\in F$.

\begin{definition}\label{D-Sigma-IE}
Let $\Sigma = \{ \sigma_i : G\to\Sym (d_i ) \}$ be a sofic approximation net for $G$.
We say that a tuple $\ox = (x_1 ,\dots ,x_k )\in X^k$ is a {\it $\Sigma$-IE-tuple} (or {\it $\Sigma$-IE-pair} in the
case $k=2$) if for every product neighbourhood $U_1 \times\dots\times U_k$ of $\ox$ the tuple $(U_1 , \dots , U_k )$
has positive upper independence density over $\Sigma$.
We say that $\ox$ is a {\it sofic IE-tuple} (or {\it sofic IE-pair} in the case $k=2$) if it is a $\Sigma$-IE-tuple
for the universal sofic approximation net $\Sigma$. We denote the $\Sigma$-IE-tuples of length $k$
by $\IE^\Sigma_k (X,G)$ and the sofic IE-tuples of length $k$ by $\IE_k^{\rm sof} (X,G)$.
\end{definition}

Note that $\IE^\Sigma_k (X,G) \subseteq \IE_k^{\rm sof} (X,G)$ for every sofic approximation net $\Sigma$.

We define $\Sigma$-IE-tuples and sofic IE-tuples of  sets in the same way as for points above.

\begin{remark}\label{R-entropy tuple}
It follows from Lemma~3.3 of \cite{KerLi07} that a nondiagonal tuple of points in $X$ is a
$\Sigma$-IE-tuple if and only if it is a $\Sigma$-entropy tuple in the sense of Section~5 in \cite{Zha12}.
In particular, if in analogy with the amenable case we define the action to have
{\it uniformly positive entropy with respect to $\Sigma$} when every nondiagonal pair in $X\times X$ is a
$\Sigma$-entropy pair, then the action has this property precisely when
every pair in $X\times X$ is a $\Sigma$-IE-pair.
\end{remark}

We will need the following consequence of Karpovsky and Milman's
generalization of the Sauer-Shelah lemma
\cite{Sau72,She72,KarMil78}.

\begin{lemma}[\cite{KarMil78}]\label{L-KM}
Given $k\geq 2$ and $\lambda > 1$ there is a constant $c>0$ such
that, for all $n\in \Nb$, if $S\subseteq \{1, 2, \dots , k \}^{\{1,
2, \dots , n\}}$ satisfies $|S|\geq ((k-1)\lambda )^n$ then there is
an $I\subseteq \{1, 2, \dots , n\}$ with $|I|\geq cn$ and $S|_I =
\{1, 2,\dots , k \}^I$.
\end{lemma}

\begin{proposition} \label{P-sofic IE to orbit IE}
A sofic IE-tuple is an orbit IE-tuple.
\end{proposition}

\begin{proof}
Fix a compatible metric $\rho$ on $X$. Let $\ox = (x_1 ,\dots ,x_k )$ be a $\Sigma$-IE-tuple and
$U_1\times \dots \times U_k$ a product neighborhood
of $\ox$. We will show that the tuple $(U_1, \dots, U_k)$ has positive independence density over $G$.

Suppose first that $k>1$.
Take $1<\lambda<\frac{k}{k-1}$. Then we have the constant $c>0$ in Lemma~\ref{L-KM}.

Let $V_1\times \dots \times V_k$ be a product neighborhood of $\ox$ such that for some $\kappa>0$
the $\kappa$-neighborhood of $V_j$ is contained in $U_j$ for all $1\le j\le k$. Then there exists a $q>0$
such that for every nonempty finite subset $F$ of $G$  and $\delta>0$ there is a cofinal set of $i$
for which the tuple $(V_1, \dots, V_k)$ has a $(\rho, F, \delta, \sigma_i)$-independence set $\cJ_i$
of cardinality at least $qd_i$.

Let $F$ be a nonempty finite subset of $G$. We will show $F$ has a subset of cardinality at least
$(cq/2)|F|$ which is an independence set for  the tuple
$(U_1, \dots, U_k)$. Let $\delta$ be a small positive number to be determined in a moment.

Take an $i$ in the above cofinal set with $|\cW_i|\ge (1-\delta)d$ for
$\cW_i:=\{a\in \{1, \dots, d_i\}: F\overset{\sigma_i(\cdot)a}{\rightarrow} \sigma_i(F)a \text{ is injective}\}$.
For each $\omega: \cJ_i\rightarrow \{1, \dots, k\}$,
take a $\varphi_\omega\in \Map (\rho ,F,\delta ,\sigma )$ such that
$\varphi_\omega(a)\in V_{\omega(a)}$ for all $a\in \cJ_i$.
Then $|\{a\in \{1, \dots, d_i\}: \rho(s\varphi_\omega(a), \varphi_\omega(sa))\le \delta^{1/2}\}|
\ge (1-\delta)d_i$ for each $s\in F$ and hence $|\Lambda_\omega|\ge (1-|F|\delta)d_i$ for
\[
\Lambda_\omega:=\{a\in \{1, \dots, d_i\}: \rho(s\varphi_\omega(a), \varphi_\omega(sa))\le \delta^{1/2} \text{ for all } s\in F\}.
\]

Set $n = |F|$. When $n\delta <1/2$,
the number of subsets of $\{1, \dots, d\}$ of cardinality no greater than $n\delta d$
is equal to $\sum_{j=0}^{\lfloor n\delta d \rfloor} \binom{d}{j}$, which is at most
$n\delta d \binom{d}{n\delta d}$,
which by Stirling's approximation is less than $\exp(\beta d)$
for some $\beta > 0$ depending on $\delta$ and $n$
but not on $d$ when $d$ is sufficiently large with $\beta\to 0$ as $\delta\to 0$ for a fixed $n$.
Thus when $\delta$ is small enough and $i$ is large enough,
there is a subset $\Omega_i$ of $\{1, \dots, k\}^{\cJ_i}$ with
$\big(\frac{k}{(k-1)\lambda}\big)^{q d_i}|\Omega_i|\ge k^{|\cJ_i|}$
such that the set $\Lambda_\omega$ is the same, say $\Theta_i$, for every $\omega \in \Omega_i$,
and $|\Theta_i|/d_i>1-|F|\delta$. Then
\[
|\Omega_i|\ge k^{|\cJ_i|}\bigg(\frac{(k-1)\lambda}{k}\bigg)^{q d_i}
\ge k^{|\cJ_i|}\bigg(\frac{(k-1)\lambda}{k}\bigg)^{|\cJ_i|}=((k-1)\lambda)^{|\cJ_i|}.
\]
By our choice of $c$, we can find a subset $\cJ'_i$ of $\cJ_i$ with $|\cJ'_i|\ge c|\cJ_i|\ge cq d_i$
such that every $\xi: \cJ'_i\rightarrow \{1, \dots, k\}$ extends to some $\omega\in \Omega_i$.

Writing $\zeta$ for the uniform probability measure on $\{1,\dots,d\}$, we have
\begin{align*}
\int_{\cW_i\cap \Theta_i} \sum_{s\in F}1_{\cJ'_i}(sa)\, d\zeta (a)
&=\sum_{s\in F}\int_{\cW_i \cap \Theta_i} 1_{\cJ'_i}(sa)\, d\zeta (a)\\
&\ge \sum_{s\in F}\bigg(\frac{|\cJ'_i|}{d_i}-(|F|+1)\delta\bigg)\ge (cq-(|F|+1)\delta) |F|,
\end{align*}
and hence $\sum_{s\in F}1_{\cJ'_i}(sa_i)\ge (cq-(|F|+1)\delta)|F|$ for some $a_i\in \cW_i\cap \Theta_i$.
Then $|J_i|\ge (cq-(|F|+1)\delta)|F|$ for $J_i:=\{s\in F: sa_i\in \cJ'_i\}$.

We claim that $J_i$ is an independence set for the tuple $(U_1, \dots, U_k)$ when $\delta<\kappa^2$.
Let $f\in \{1, \dots, k\}^{J_i}$.
Since $a_i\in \cW_i$, the map $J_i\overset{\sigma_i(\cdot)a_i}{\rightarrow} \sigma_i(J_i)a_i$ is bijective.
Thus we can define $\xi'\in \{1, \dots, k\}^{\sigma(J_i)a_i}$ by $\xi'(sa_i)=f(s)$ for $s\in J_i$.
Extend $\xi'$ to some $\xi\in \{1, \dots, k\}^{\cJ'_i}$. Then we can extend $\xi$ to some
$\omega \in \Omega_i$. For every $s\in J_i$, since $sa_i\in \cJ_i$ and $a_i\in \Theta_i=\Lambda_\omega$, we have
$\varphi_\omega(s a_i)\in V_{\omega (s a_i)}=V_{f(s)}$ and
$\rho(s\varphi_\omega(a_i), \varphi_\omega(s a_i))\le \delta^{1/2}<\kappa$.
By the choice of $\kappa$ we have $s \varphi_\omega(a_i)\in U_{f(s)}$. This proves our claim.

Taking $\delta$ to be small enough, we have $|J_i|\ge (cq-(|F|+1)\delta)|F|\ge (cq/2)|F|$ as desired.

The case $k=1$ can be established by a simpler version of the above argument that considers only a single
map of the form $\varphi_\omega$ and does not require the invocation of the constant $c$ from
Lemma~\ref{L-KM} and the associated use of Stirling's approximation.
\end{proof}

For the remainder of this subsection, $\Sigma = \{ \sigma_i :G\to\Sym (d_i) \}$ is a fixed
but arbitrary sofic approximation sequence.

In the case that $G$ is amenable, the independence density $I(\oA)$ of a tuple $\oA$ of subsets of $X$
was defined on page 887 of \cite{KerLi07} as the limit of $\varphi_{\oA}(F) /|F|$ as
the nonempty finite set $F\subseteq G$ becomes more and more left invariant,
where $\varphi_{\oA}(F)$ denotes the maximum of the cardinalities of the independence sets
for $\oA$ which are contained in $F$.
A tuple $(x_1 , \dots , x_k ) \in X^k$ is an {\it IE-tuple} if for every product neighbourhood
$U_1 \times\cdots\times U_k$ of $(x_1 , \dots , x_k )$ the independence density $I(\oU)$
of the tuple $\oU = (U_1 ,\dots, U_k )$ is positive.

To establish Theorem~\ref{T-amenable case: orbit IE=IE} we need the following version of the Rokhlin lemma
for sofic approximations, which appears as Lemma~4.6 in \cite{KerLi10}. For $\lambda\geq 0$,
a collection of subsets of $\{ 1,\dots ,d\}$ is said to {\it $\lambda$-cover} $\{1, \dots, d\}$ if
its union has cardinality at least $\lambda d$.

\begin{lemma}\label{L-Rokhlin2}
Let $G$ be a countable amenable discrete group.
Let $0\le \tau<1$ and $0<\eta<1$. Let $K$ be a nonempty finite subset of $G$ and $\delta>0$.
Then there are an $\ell\in \Nb$, nonempty finite subsets $F_1, \dots, F_\ell$ of $G$ with $|KF_k \setminus F_k|<\delta |F_k|$ and
$|F_kK\setminus F_k|<\delta|F_k|$ for
all $k=1, \dots, \ell$, a finite set $F\subseteq G$ containing $e$,
and an $\eta'>0$ such that,
for every $d\in \Nb$, every map $\sigma: G\rightarrow \Sym(d)$ for which there is a set $B\subseteq \{1, \dots, d\}$ satisfying
$|B|\ge (1-\eta')d$ and
\[
\sigma_{st}(a)=\sigma_s\sigma_t(a), \sigma_s(a)\neq \sigma_{s'}(a), \sigma_e(a)=a
\]
for all $a\in B$ and $s, t, s'\in F$ with $s\neq s'$, and every set $V\subseteq \{1, \dots, d\}$ with $|V|\ge (1-\tau)d$,
there exist $C_1, \dots, C_\ell\subseteq V$ such that
\begin{enumerate}
\item for every $k=1, \dots, \ell$, the map $(s, c)\mapsto \sigma_s(c)$ from $F_k\times C_k$ to $\sigma(F_k)C_k$ is bijective,

\item the family $\{ \sigma(F_1)C_1, \dots, \sigma(F_\ell)C_\ell \}$ is disjoint and $(1-\tau-\eta)$-covers $\{1, \dots, d\}$.
\end{enumerate}
\end{lemma}

\begin{theorem} \label{T-amenable case: orbit IE=IE}
Suppose that $G$ is amenable. Then IE-tuples, orbit IE-tuples, $\Sigma$-IE-tuples, and sofic IE-tuples
are all the same thing.
\end{theorem}

\begin{proof}
By Proposition~\ref{P-sofic IE to orbit IE}, sofic IE-tuples and $\Sigma$-IE-tuples
are orbit IE-tuples. That orbit IE-tuples are IE-tuples is clear
in view of the definition of the independence density $I(\oA )$ of a tuple $\oA$ of subsets
of $X$. It thus remains to show that IE-tuples are both sofic IE-tuples and $\Sigma$-IE-tuples.

To prove that IE-tuples are $\Sigma$-IE-tuples,
it suffices to demonstrate that, given
a tuple $\oU = (U_1,\dots,U_k)$ of subsets of $X$ with $I(\oU) > 0$,
the tuple $\oU$ has positive upper independence density over $\Sigma$.
Set $\lambda = I(\oU) > 0$. Let $F$ be a nonempty finite subset of $G$ and $\delta > 0$.

Let $\eta>0$, to be determined.
By Lemma~\ref{L-Rokhlin2} we can find an $\ell\in\Nb$ and nonempty finite sets $F_1 ,\dots ,F_\ell \subseteq G$
such that
(i) the sets $F_1,\dots,F_\ell$ are sufficiently left invariant
so that for each $i=1,\dots ,\ell$ there is a set $J_i \subseteq F_i$ which is an independence set for $\oU$
and has cardinality at least $\lambda |F_i|/2$, and (ii)
for every good enough sofic approximation $\sigma : G\to\Sym(d)$ there exist
$C_1,\dots,C_\ell \subseteq \{ 1,\dots ,d \}$ satisfying the following:
\begin{enumerate}
\item for every $i=1, \dots, \ell$ and $c\in C_i$, the map $s\mapsto \sigma_s(c)$ from $F_i$ to $\sigma(F_i)c$ is bijective,

\item the family of sets $\sigma(F_i)c$ for $i=1,\dots,\ell$ and $c\in C_i$
is disjoint and $(1-\eta)$-covers $\{1, \dots, d\}$.
\end{enumerate}

Let $\sigma:G\rightarrow \Sym (d)$ be a sufficiently good sofic approximation for $G$ for some $d\in\Nb$.
For every $h = (h_1,\dots,h_\ell )\in\prod_{i=1}^\ell X^{C_i}$ take a map
$\varphi_h : \{ 1,\dots ,d\} \to X$ such that
\[
\varphi_h (sc) = s(h_i (c))
\]
for all $i\in \{1,\dots,\ell\}$, $c\in C_i$, and $s\in F_i$. We may assume
in our invocation of Lemma~\ref{L-Rokhlin2} above that the sets $F_1,\dots,F_\ell$
are sufficiently left invariant so that, assuming that $\eta$ is sufficiently small
and $\sigma$ is a sufficiently good sofic approximation, we
have $\varphi_h \in\Map (\rho,F,\delta,\sigma)$ for every
$h\in\prod_{i=1}^\ell X^{C_i}$.
Write $\cJ$ for the subset
$\bigcup_{i=1}^\ell \bigcup_{s\in J_i} \bigcup_{c\in C_i} \sigma_s (c)$ of $\{1,\dots,d\}$.
From (1) and (2) we obtain
\[
|\cJ|= \sum_{i=1}^\ell |J_i| |C_i|\geq\sum_{i=1}^\ell \frac{\lambda}{2} |F_i| |C_i|
\geq \frac{\lambda}{2} (1-\eta)d \geq \frac{\lambda}{4} d
\]
assuming that $\eta\leq 1/2$.
Now whenever we are given $\omega_i \in \{ 1,\dots ,k\}^{J_i}$ for $i=1,\dots,\ell$
we can find, since each $J_i$ is an independence set for $\oU$, an
$h = (h_1,\dots,h_\ell)\in\prod_{i=1}^\ell X^{C_i}$ such that
$sh_i (c) \in U_{\omega_i (s)}$ for all $i=1,\dots,\ell$ and $s\in J_i$.
The maps $\varphi_h$ for such $h$ then witness the fact that $\cJ$ is a
$(\rho,F,\delta,\sigma)$-independence set for $\oU$.
It follows that $\oU$ has positive upper independence density over $\Sigma$.
Hence IE-tuples are $\Sigma$-IE-tuples.

The above argument also shows that IE-tuples are sofic IE-tuples, and so we are done.

One can also give the following direct proof that IE-tuples are orbit IE-tuples.
It suffices to show that if $\oA = (A_1 , \dots ,A_k )$ is a
tuple of subsets of $X$ and $q>0$ is such that for every nonempty finite subset $K$ of $G$ and
$\eps>0$ there exist a nonempty finite subset
$F$ of $G$ with $|KF\setminus F|\le \eps |F|$ and a $J\subseteq F$ with $|J|\ge q|F|$
which is an independence set for $\oA$,
then the independence density of $\oA$ over $G$ is at least $q$.
Let $F_1$ be a nonempty finite subset of $G$. Let $1>\delta>0$.
Take $\eps>0$ be a small number which we shall determine in a moment.
Then there exist a nonempty finite subset
$F$ of $G$ with $|F_1^{-1}F\setminus F|\le \eps |F|$ and a $J\subseteq F$ with $|J|\ge q|F|$ which is an independence set for $\oA$.
Set $F'=\{s\in F: F_1^{-1}s\subseteq F\}$. Taking $\eps$ to be small enough, we have $|F'|\ge (1-\delta)|F|$.
Note that the function $\sum_{s\in F}1_{F_1s}$ has value $|F_1|$ at every point of  $F'$.
Thus
\begin{align*}
\sum_{t\in J\cap F'}\sum_{s\in F}1_{F_1s}(t)=|J\cap F'||F_1|.
\end{align*}
We also have
\begin{align*}
\sum_{t\in J\cap F'}\sum_{s\in F}1_{F_1s}(t)=\sum_{s\in F}\sum_{t\in J\cap F'}1_{F_1s}(t)=\sum_{s\in F}|F_1s\cap (J\cap F')|.
\end{align*}
Therefore we can find an $s\in F$ with
\begin{align*}
|F_1s\cap (J\cap F')|\ge \frac {|J\cap F'||F_1|}{|F|}\ge (q-\delta)|F_1|.
\end{align*}
Since $F_1\cap (J\cap F')s^{-1}$ is an independence set for $\oA$, we deduce that $F_1$ has a subset
of cardinality at least $(q-\delta)|F_1|$ which is an independence set for $\oA$.
Letting $\delta\to 0$, we get that $F_1$ has a subset of cardinality at least $q|F_1|$
which is an independence set for $\oA$.
Therefore the independence density $I(\oA)$ of $\oA$ is at least $q$, and so we conclude that
IE-tuples are orbit IE-tuples.
\end{proof}

The surprising fact above is that IE-tuples are orbit IE-tuples in the amenable case.
It is clear however for a Bernoulli action that all tuples are orbit IE-tuples.
Notice also that the argument above works equally well if in the definition of
$\Sigma$-IE-tuples we use positive upper independence density over $\Sigma$
with respect to a fixed free ultrafilter $\fF$.

\begin{remark}
The product formula for $\IE$-tuples as defined in the amenable framework
was established in Theorem~3.15 of \cite{KerLi07} using a measure-theoretic argument.
We can now combine Theorems~\ref{T-amenable case: orbit IE=IE} and \ref{T-product for orbit IE}
to obtain a combinatorial proof.
\end{remark}

\begin{remark}\label{R-density}
The proof of Theorem~\ref{T-amenable case: orbit IE=IE} shows that the independence density
$I(\oA)$, as defined on page 887 of \cite{KerLi07} and recalled before the theorem statement,
coincides with the independence density defined in
Definition~\ref{D-independence density}. We may thus use the notation $I(\oA)$ without ambiguity
to denote the more general independence density of Definition~\ref{D-independence density}.
\end{remark}

\begin{remark}
When $G$ is amenable, it is clear from the classical $(n,\eps)$-separated set formulation of
topological entropy that the entropy of an action $G\curvearrowright X$
is bounded below by the supremum of $I(\oA) \log k$ over all pairs $(k,\oA)$
where $k\in\Nb$ and $\oA$ is a $k$-tuple of pairwise disjoint closed subsets of $X$.
For Bernoulli actions the two quantities are equal.
In the nonamenable case, the entropy fails in general to be bounded below by $\sup_{(k,\oA)} I(\oA) \log k$,
where $I(\oA)$ is as defined in Remark~\ref{R-density}.
Indeed an example of Ornstein and Weiss \cite[Appendix C]{OrnWei87} shows that the Bernoulli action
$F_2\curvearrowright \{ 0,1 \}^{F_2}$ over the free group on two generators has Bernoulli factors
over arbitrarily large finite sets of symbols, in which case the supremum is infinite.
\end{remark}

We next aim to establish some basic properties of $\Sigma$-IE-tuples in Proposition~\ref{P-basic}.

From Lemma~3.6 of \cite{KerLi07} we obtain:

\begin{lemma}\label{L-decomposition indep}
Let $k\in \Nb$. Then there is a constant $c>0$ depending only on $k$ with the following property.
Let $\oA=(A_1, \dots, A_k)$ be a $k$-tuple of subsets of
$X$ and suppose $A_1 = A_{1, 1}\cup A_{1,2}$.
Let $F$ be a nonempty finite subset of $G$ and $\delta > 0$.
Let $\sigma$ be a map from $G$ to $\Sym (d)$ for some $d\in\Nb$.
If a set $J\subseteq \{ 1,\dots ,d\}$ is a  $(\rho ,F,\delta ,\sigma )$-independence set for $\oA$, then there exists an $I\subseteq J$
such that $|I|\ge c|J|$ and $I$ is a  $(\rho ,F,\delta ,\sigma )$-independence set for
$(A_{1,1}, \dots, A_k)$ or $(A_{1,2}, \dots, A_k)$.
\end{lemma}

From Lemma~\ref{L-decomposition indep} we get:

\begin{lemma}\label{L-decomposition E}
Let $\oA= (A_1, \dots, A_k )$ be a $k$-tuple of subsets of $X$ which has
positive upper independence density over $\Sigma$. Suppose that $A_1 = A_{1,
1}\cup A_{1, 2}$. Then at least one of the tuples $(A_{1,1},\dots,
A_k)$ and $(A_{1,2}, \dots, A_k)$ has positive upper  independence density over $\Sigma$.
\end{lemma}

\begin{lemma} \label{L-positive entropy to independence}
$h_\Sigma(X, G)>0$ if and only if there are disjoint closed subsets $A_0$ and $A_1$ of $X$
such that $(A_0, A_1)$ has positive upper  independence density over $\Sigma$.
\end{lemma}
\begin{proof} Let $\rho$ be a compatible metric on $X$ with $\diam_\rho(X)\le 1$.
Then $h_{\Sigma, \infty}(\rho)=h_\Sigma(X, G)$. The ``if'' part is obvious.
So assume $h_{\Sigma, \infty}(\rho)>0$.
Then $h^{6\eps}_{\Sigma, \infty}(\rho)>0$ for some $\eps>0$. Set $c=h^{6\eps}_{\Sigma, \infty}(\rho)/2$.

Take a finite $(\rho, 2\eps)$-dense subset $Z$ of $X$.
Consider on $X$ the continuous pseudometrics $\rho^z$, for $z\in Z$, and $\rho'$ given by
\[ \rho^z(x, y)=|\rho(x, z)-\rho(y, z)|, \hspace*{5mm} \rho'(x, y)=\max_{z\in Z}\rho^z(x, y).\]
Note that if $\rho(x, y)\ge 6\eps$ for some $x, y\in X$, then $\rho'(x, y)\ge 2\eps$.
It follows that if
$d\in \Nb$ and $\varphi$ and $\psi$ are maps from $\{1, \dots, d\}$ to $X$
with $\rho_{\infty}(\varphi, \psi)\ge 6\eps$,
then $\rho'_{\infty}(\varphi, \psi)\ge 2\eps$.

Take an increasing sequence $\{F_n\}_{n\in \Nb}$ of nonempty finite subsets of $G$ with union $G$
and a decreasing sequence $\{\delta_n\}_{n\in \Nb}$ of positive numbers converging to $0$.
For each $n\in \Nb$, there is a cofinal set $I_n$ of $i$ for which one has
$N_{6\eps}(\Map(\rho, F_n, \delta_n, \sigma_i), \rho_\infty)\ge \exp(c d_i)$.
Then $N_{2\eps}(\Map(\rho, F_n, \delta_n, \sigma_i), \rho'_\infty)\ge \exp(c d_i)$
for all $i\in I_n$. For each $i\in I_n$ and $z\in Z$ take
a $(\rho^z_\infty, \eps)$-separated subset
$W_{i, z}$ of $\Map(\rho, F_n, \delta_n, \sigma_i)$ of maximum cardinality. Then
\[
N_{2\eps}(\Map(\rho, F_n, \delta_n, \sigma_i), \rho'_\infty)
\le \prod_{z\in Z}|W_{i, z}|
= \prod_{z\in Z}N_\eps(\Map(\rho, F_n, \delta_n, \sigma_i), \rho^z_\infty).
\]
Thus $N_\eps(\Map(\rho, F_n, \delta_n, \sigma_i), \rho^{z_{n, i}}_\infty)\ge \exp(c d_i/|Z|)$
for some $z_{n, i}\in Z$. Replacing $I_n$ by a confinal subset if necessary,
we may assume that $z_{n, i}$ is the same, say $z_n$, for all $i\in I_n$.
Passing to a subsequence of $\{(F_n, \delta_n)\}_{n \in \Nb}$ if necessary,
we may assume that $z_n$ is the same, say $\mathfrak{z}$, for all $n\in \Nb$.

Note that if $W$ is a $(\rho^{\mathfrak{z}}_\infty, \eps)$-separated subset of
$\Map(\rho, F_n, \delta_n, \sigma_i)$, then
the set $\{ \rho(\mathfrak{z}, \cdot)\circ  \varphi: \varphi \in W\}$ in $\ell_{\infty}^{d_i}$ is
$(\|\cdot \|_\infty, \eps)$-separated.
By \cite[Lemma 2.3]{GlaWei95}, there are constants $c'$ and $\delta>0$ depending only on $c/|Z|$
and $\eps$ such that for every $n\in \Nb$ and large enough
$i\in I_n$ there are a $t_{n, i}\in [0, 1]$ and a subset $J_{n, i}$ of $\{1, \dots, d_i\}$
with $|J_{n, i}|\ge c'd_i$ so that for every $\omega: J_{n, i}\rightarrow \{0, 1\}$
there are a $\varphi_\omega\in \Map(\rho, F_n, \delta_n, \sigma_i)$ such that for
all $a\in J_{n, i}$ we have $\rho(\mathfrak{z}, \varphi_\omega(a))\ge t_{n, i}+\delta$ or
$\rho(\mathfrak{z}, \varphi_\omega(a))\le t_{n, i}-\delta$ depending on
whether $\omega(a)=0$ or $\omega(a)=1$. Replacing $I_n$ by a confinal subset if necessary,
we may assume that there is a $t_n\in [0, 1]$ such that $|t_{n, i}-t_n|<\delta/4$ for all $i\in I_n$.
Replacing $\{F_n, \delta_n\}_{n\in \Nb}$ by a subsequence if necessary, we may assume that
there is a $t\in [0, 1]$ such that $|t_n-t|<\delta/4$ for all $n\in \Nb$.
Set $A_0=\{x\in X: \rho(\mathfrak{z}, x)\ge t+\delta/2\}$ and
$A_1=\{x\in X: \rho(\mathfrak{z}, z)\le t-\delta/2\}$.
Then for every $n\in \Nb$ and $i\in I_n$, the set $J_{n, i}$ is a
$(\rho, F_n, \delta_n, \sigma_i)$-independence set for $(A_0, A_1)$.
Thus $(A_0, A_1)$ has positive upper independence density over $\Sigma$.
\end{proof}

The following is obvious.

\begin{lemma} \label{L-nonnegative entropy}
$h_\Sigma(X, G)\ge 0$ if and only if $X$ as a $1$-tuple has positive upper independence density over $\Sigma$.
\end{lemma}

\begin{proposition}\label{P-basic}
The following are true:
\begin{enumerate}
\item Let $(A_1, \dots , A_k )$ be a tuple of closed subsets of $X$ which has positive upper
independence density over $\Sigma$. Then there exists a $\Sigma$-IE-tuple $(x_1,\dots ,
x_k)$ with $x_j\in A_j$ for all $1\le j\le k$.

\item $\IE_1^\Sigma(X, G)$ is nonempty if and only if $h_\Sigma(X, G)\ge 0$.

\item $\IE_2^\Sigma(X, G)\setminus \Delta_2(X)$ is nonempty if and only if $h_\Sigma(X, G)>0$,
where $\Delta_2 (X)$ denotes the diagonal in $X^2$.

\item $\IE_k^\Sigma(X, G)$ is a closed subset of $X^k$ which is invariant under the product action.

\item Let $\pi:(X, G)\rightarrow (Y,G)$ be a factor map. Then
$(\pi\times\cdots\times \pi )(\IE_k^\Sigma(X, G))\subseteq \IE_k^\Sigma(Y, G)$.

\item Suppose that $Z$ is a closed $G$-invariant subset of $X$. Then
$\IE_k^\Sigma(Z, G )\subseteq \IE_k^\Sigma(X, G)$.
\end{enumerate}
\end{proposition}

\begin{proof}
Assertion (1) follows from Lemma~\ref{L-decomposition E} and a
simple compactness argument.
Assertion (2) follows from assertion (1) and Lemma~\ref{L-nonnegative entropy}.
Assertion (3)
follows directly from assertion (1) and Lemma~\ref{L-positive entropy to independence}.
Assertion (4) follows from the observation that, given a compatible metric $\rho$ of $X$,
for any $s\in G$, nonempty finite subset $F$ of $G$, and $\delta>0$ there is a $\delta'>0$ such
that, for every $d\in \Nb$ and map $\sigma: G\rightarrow \Sym(d)$
which is a good enough sofic approximation for $G$,
if $\varphi \in \Map(\rho, \{s^{-1}\}\cup (s^{-1}F), \delta', \sigma)$,
then $\alpha_s\circ \varphi\circ \sigma_{s^{-1}}\in \Map(\rho, F, \delta, \sigma)$,
where $\alpha_s$ is the transformation $x\mapsto sx$ of $X$.
Assertions (5) and (6) are trivial.
\end{proof}

\begin{remark}
The inclusion in (5) above is an equality when $G$ is amenable, since $\Sigma$-IE-tuples
are the same as IE-tuples by Theorem~\ref{T-amenable case: orbit IE=IE}. Equality can fail
however if $G$ is nonamenable: Take an action $G\curvearrowright X$ with $h_\Sigma (X,G)=-\infty$
and an action $G\curvearrowright Y$ with $h_\Sigma (Y,G) > 0$. Then $G\curvearrowright Y$
has a nondiagonal $\Sigma$-IE-pair, while the product action
$G\curvearrowright X\times Y$, which factors onto $G\curvearrowright Y$ via the second coordinate projection,
satisfies $h_\Sigma (X\times Y,G)=-\infty$ and hence has no
nondiagonal $\Sigma$-IE-pairs.
\end{remark}

\begin{remark}
The analogue for orbit IE-tuples of the localization in Proposition~\ref{P-basic}(1) does not hold
in the nonamenable case. Indeed for any action $G\curvearrowright X$ of a discrete group the $1$-tuple $X$
has positive independence density, while the boundary action $F_2 \curvearrowright \partial F_2$
of the free group on two generators (where $\partial F_2$ consists of infinite reduced words in the
standard generators and their inverses, with the action by left concatenation and reduction)
is easily seen not to admit any orbit IE-$1$-tuples.
\end{remark}

From Proposition~\ref{P-basic}(5) we get the following. As in Theorem~\ref{T-product for orbit IE},
the inclusion below is understood with respect to the identification of
$((x_1,\dots,x_k),(y_1,\dots,y_k))\in X^k \times Y^k$ and $((x_1,y_1),\dots,(x_k,y_k))\in (X\times Y)^k$.

\begin{proposition} \label{P-product for IE}
$\IE^\Sigma_k (X\times Y,G)\subseteq \IE^\Sigma_k (X,G) \times \IE^\Sigma_k (Y,G)$.
\end{proposition}

The problem of the reverse inclusion will be taken up in the next section.

For the remainder of this section $X$ is the unit ball of $\ell^p(G)$ for
some $1\le p<\infty$ equipped with the pointwise convergence topology,
and the action $G\curvearrowright X$ is by left shifts.
We will use some of the above results to compute the
sofic topological entropy of this action to be zero when $G$ is infinite.

Recall from the end of Section~\ref{S-orbit} that a tuple $\ox = (x_1 , \dots , x_k )\in X^k$ is an IN-tuple
if for every product neighbourhood $U_1 \times\dots\times U_k$ of $\ox$ the tuple
$(U_1 ,\dots, U_k )$ has arbitrarily large finite independence sets.
We write $\IN_k (X,G)$ for the set of IN-tuples of length $k$.

\begin{lemma} \label{L-unit ball null}
For every $k\in\Nb$ the set $\IN_k(X,G)$ consists of the single element $(0, \dots, 0)$.
\end{lemma}

\begin{proof}
Clearly $(0, \dots, 0)\in \IN_k(X)$ for every $k\in \Nb$.
Also note that if $\ox=(x_1,\dots, x_k)\in \IN_k(X)$ then $x_1, \dots, x_k\in \IN_1(X)$.
Thus it suffices to show $\IN_1(X)\subseteq \{0\}$.

Let $x\in X$ with $x\neq 0$. Then $(tx)_e\neq 0$ for some $t\in G$. Set $r=|(tx)_e|/2>0$ and $U=\{y\in X: |y_e|\ge r\}$.
Then $U$ is a neighborhood of $tx$ in $X$.
Let $F\subseteq G$ be a finite independence set for $U$. Then $\bigcap_{s\in F} s^{-1}U$ is nonempty.
Take $y\in \bigcap_{s\in F} s^{-1}U$.
Then $sy\in U$ and hence $|y_{s^{-1}}|=|(sy)_e|\ge r$ for every $s\in F$. It follows that
\[
|F|r^p\le \sum_{s\in F}|y_{s^{-1}}|^p\le \|y\|_p^p\le 1,
\]
and hence $|F|\le r^{-p}$. Therefore $tx\not\in \IN_1(X)$. Since $\IN_1(X)$ is $G$-invariant, $x\not\in \IN_1(X)$.
\end{proof}

\begin{proposition} \label{P-unit ball zero entropy}
Suppose that $G$ is infinite. Then $h_\Sigma(X, G)=0$.
\end{proposition}
\begin{proof}
By Lemma~\ref{L-unit ball null}, Propositions~\ref{P-basic}(2), and Propositions~\ref{P-sofic IE to orbit IE}
and \ref{P-orbit IE to IN}, we have $h_\Sigma(X, G)\le 0$.
Since $X$ has the fixed point $0$, we have $h_\Sigma(X, G)\ge 0$.
Therefore $h_\Sigma(X, G)=0$.
\end{proof}

\section{Product Formula for IE-tuples}\label{S-product}

In order to hope for a product formula for $\Sigma$-IE-tuples beyond the amenable case,
we must be able to witness independence density in some uniform way,
in analogy with the definition of orbit IE-tuples in Section~\ref{S-orbit}
(see Theorem~\ref{T-product for orbit IE}).
This can be achieved
by taking a free ultrafilter $\fF$ on $\Nb$ and requiring that the independence sets
in Definition~\ref{D-positive independence density} exist for a set of $i$ belonging to $\fF$
instead a cofinal set of $i$.
Thus for the purposes of this section we fix a free ultrafilter $\fF$ on $\Nb$ and switch to
definition of $\Sigma$-IE-tuples based on this interpretation of positive density.
We will similarly understand sofic
topological entropy to be defined by using an ultralimit over $\fF$ in Definition~\ref{D-topological entropy}
instead of the limit supremum. We do not know whether our product formula results,
Proposition~\ref{P-entropy for product} and Theorem~\ref{T-ergodic to product},
hold for the original definitions.

For the first part of our discussion, up to and including Lemma~\ref{L-tuple of product sets}, $G$ is a
countable sofic group and $\Sigma=\{\sigma_i: G\rightarrow \Sym(d_i)\}_{i=1}^\infty$ a fixed but arbitrary
sofic approximation sequence for $G$.

\begin{proposition}\label{P-entropy for product}
Let $G$ act continuously on compact metrizable spaces $X$ and $Y$. Then
\[
h_\Sigma(X\times Y, G)=h_\Sigma(X, G)+h_\Sigma(Y, G).
\]
\end{proposition}
\begin{proof} Fix compatible metrics $\rho^X$ and $\rho^Y$ on $X$ and $Y$ respectively.
Define a compatible metric $\rho_{X\times Y}$ on $X\times Y$ by
\[
\rho^{X\times Y}((x_1, y_1), (x_2, y_2))=\rho^X(x_1, x_2)+\rho^Y(y_1, y_2)
\]
for $(x_1, y_1), (x_2, y_2)\in X\times Y$.

Let $d\in \Nb$. Identify $(X\times Y)^{\{1, \dots, d\}}$ with $X^{\{1, \dots, d\}}\times Y^{\{1, \dots, d\}}$ naturally.
Note that for all $\varphi, \varphi'\in X^{\{1, \dots, d\}}$ and $\psi, \psi'\in Y^{\{1, \dots, d\}}$ one has
\[
\max\big(\rho^X_2(\varphi, \varphi'), \rho^Y_2(\psi, \psi' )\big)
\le \rho^{X\times Y}_2((\varphi, \psi), (\varphi', \psi'))
\le \rho^X_2(\varphi, \varphi')+\rho^Y_2(\psi, \psi').
\]
Let $F$ be a nonempty finite subset of $G$,  $\delta>0$, and $\eps>0$, and let
$\sigma$ be  a map from $G$ to $\Sym(d)$.
Then $\Map(\rho^X, F, \delta, \sigma)\times \Map(\rho^Y, F, \delta, \sigma)\subseteq \Map(\rho^{X\times Y}, F, 2\delta, \sigma)$.
Furthermore, for every $(\rho^X_2, \eps)$-separated subset $\sW_X$ of $\Map(\rho^X, F, \delta, \sigma)$
and every $(\rho^Y_2, \eps)$-separated subset $\sW_Y$ of $\Map(\rho^X, F, \delta, \sigma)$,
the set $\sW_X\times \sW_Y$ is $(\rho^{X\times Y}_2, \eps)$-separated.
It follows that $h^\eps_{\Sigma, 2}(\rho^{X\times Y}, F, 2\delta)\ge h^\eps_{\Sigma, 2}(\rho^X, F, \delta)+h^\eps_{\Sigma, 2}(\rho^Y, F, \delta)$,
and hence $h_\Sigma(X\times Y, G)\ge h_\Sigma(X, G)+h_\Sigma(Y, G)$.

Note that for any $(\rho^X_2, \eps)$-spanning subset $\sW_X$ of $\Map(\rho^X, F, \delta, \sigma)$
and any $(\rho^Y_2, \eps)$-spanning subset $\sW_Y$ of $\Map(\rho^X, F, \delta, \sigma)$,
the set $\sW_X\times \sW_Y$ is $(\rho^{X\times Y}_2, 2\eps)$-spanning for (though not necessarily contained in)
$\Map(\rho^{X\times Y}, F, \delta, \sigma)$.
It follows that $N_{4\eps}(\rho^{X\times Y}, F, \delta, \sigma)\le N_\eps(\rho^X, F, \delta, \sigma)\times N_\eps(\rho^Y, F, \delta, \sigma)$,
and hence $h^{4\eps}_{\Sigma, 2}(\rho^{X\times Y}, F, \delta)\le h^\eps_{\Sigma, 2}(\rho^X, F, \delta)+h^\eps_{\Sigma, 2}(\rho^Y, F, \delta)$.
Consequently, $h_\Sigma(X\times Y, G)\le h_\Sigma(X, G)+h_\Sigma(Y, G)$.
\end{proof}

The Loeb space and the Loeb measure were introduced by Loeb in \cite{Loe75}. An exposition can be found in \cite{AleGleGor99}.
The Loeb space is the ultraproduct space $\prod_\fF\{1, \dots, d_i\}$. A subset $Y$ of $\prod_\fF\{1, \dots, d_i\}$ is called {\it internal}
if it is of the form $\prod_\fF Y_i$ for a sequence $\{Y_i\}_{i\in \Nb}$ with $Y_i\subseteq \{1, \dots, d_i\}$ for all $i\in \Nb$.
The collection $\fI$ of inner subsets is an algebra. The Loeb measure is the unique probability measure $\mu$ on the $\sigma$-algebra $\fB$
generated by $\fI$ such that $\mu(Y)=\lim_{i\to \fF} |\cY_i|/d_i$ for every internal set $Y=\prod_\fF \cY_i$.
For every $Z\in \fB$ there exists a $Y\in \fI$ such that $\mu(Y\Delta Z)=0$.

For each $d\in \Nb$, denote by $\rho_{\rm Hamm}$ the normalized Hamming distance on $\Sym(d)$ defined by
\[
\rho_{\rm Hamm}(\tau, \tau')=\frac{1}{d}|\{a\in \{1, \dots, d\}: \tau(a)\neq \tau'(a)\}|.
\]

The ultraproduct group $\prod_\fF\Sym(d_i)$ has a natural action on $\prod_\fF\{1, \dots, d_i\}$ preserving $\mu$. One has
a bi-invariant pseudometric $\rho_L$ on $\prod_\fF\Sym(d_i)$ defined by
$\rho_L(\tau, \tau')=\mu(\{y\in \prod_\fF\{1, \dots, d_i\}: \tau y\neq \tau'y\})$.
For any $\tau=(\tau_i)_i, \tau'=(\tau'_i)_i\in \prod_\fF\Sym(d_i)$ with $\tau_i, \tau'_i\in \Sym(d_i)$ for all $i\in \Nb$,  one has
$\rho_L(\tau, \tau')=\lim_{i\to \fF}\rho_{\rm Hamm}(\tau_i, \tau'_i)$.
Denote by $\fG$ the quotient group of $\prod_\fF\Sym(d_i)$ by $\rho_L$.
Then we may think of $\fG$ as acting on $\prod_\fF \{1, \dots, d_i\}$ by $\mu$-preserving transformations.

The sofic approximation sequence $\Sigma$ gives rise to a natural group embedding of $G$ into $\fG$.
Thus we may think of $G$ as a subgroup of $\fG$. Denote by $G'$ the subgroup of $\fG$ consisting of elements commuting with $G$.

As before, the equality below is understood with respect to the identification of
$((x_1,\dots,x_k),(y_1,\dots,y_k))\in X^k \times Y^k$ and $((x_1,y_1),\dots,(x_k,y_k))\in (X\times Y)^k$.

\begin{theorem}\label{T-ergodic to product}
Suppose that the action of $G'$ on $(\prod_\fF\{1, \dots, d_i\}, \fB, \mu)$ is ergodic.
Let $G$ act continuously on compact metrizable spaces $X$ and $Y$. Let $k\in \Nb$. Then
\[
\IE^\Sigma_k (X\times Y,G)=\IE^\Sigma_k (X,G) \times \IE^\Sigma_k (Y,G).
\]
\end{theorem}

We prove the theorem by way of the following results.

\begin{definition}\label{D-independence set on Loeb space}
Let $G$ act continuously on a compact metrizable space $X$. Let $\rho$ be a dynamically generating continuous pseudometric on $X$.
Let $\oA = (A_1 , \dots ,A_k )$ be a tuple of subsets of $X$.
We say that an internal set $Y=\prod_\fF Y_i$ with $Y_i\subseteq \{1, \dots, d_i\}$ for all $i\in \Nb$ is an {\it independence set} for
$\oA$ if for every nonempty finite subset $F$ of $G$ and every $\delta>0$ the set of all $i\in \Nb$ for which $Y_i$
is a $(\rho, F, \delta, \sigma_i)$-independence set for $\oA$ belongs to $\fF$.
\end{definition}

From Lemma~\ref{L-change pseudometric} it is easy to see that Definition~\ref{D-independence set on Loeb space}
does not depend on the choice of $\rho$.

Consistent with our interpretation of the equality in Theorem~\ref{T-ergodic to product},
in Proposition~\ref{P-basics of internal independence sets} and
Lemma~\ref{L-tuple of product sets}  we understand $\oA \times \oB$ to mean $(A_1 \times B_1 ,\dots , A_k \times B_k)$
where $\oA = (A_1 ,\dots ,A_k)$ and $\oB = (B_1 ,\dots ,B_k)$.

\begin{proposition} \label{P-basics of internal independence sets}
Let $G$ act continuously on compact metrizable spaces $X$ and $Y$. Let $\oA$ and $\oB$ be $k$-tuples of subsets of
$X$ and $Y$ respectively for some $k\in \Nb$. Then the following hold:
\begin{enumerate}
\item $\oA$ has positive upper independence density over $\Sigma$ if and only if $\oA$ has an internal independence
set $Z$ with $\mu(Z)>0$.

\item The set of internal independence sets for $\oA$ is $G'$-invariant.

\item An internal set is an independence set for $\oA \times \oB$ if and only if it is an independence set for both $\oA$ and $\oB$.
\end{enumerate}
\end{proposition}

\begin{proof}
Fix a compatible metric $\rho$ on $X$ which gives $X$ diameter at most $1$.

(1). The ``if'' part is obvious. Suppose that $\oA$ has positive upper independence density over $\Sigma$.
Let $q>0$ be as in Definition~\ref{D-positive independence density}. Let $\{F_n\}_{n\in \Nb}$ be an increasing
sequence of finite subsets of $G$ with $\bigcup_{n\in \Nb} F_n=G$.
For each $n\in \Nb$, denote by $W'_n$ the set of all $i\in \Nb$ for which there is a $(\rho_X, F_n, 1/n, \sigma_i)$-independence
set $Z_i$ for $\oA$ with $\zeta(Z_i)\ge q$. Also set $W_n=W'_n\setminus \{1, \dots, n-1\}$.
Then $W'_n\in \fF$ by our assumption and hence $W_n\in \fF$ for each $n\in \Nb$.
Note that the sequence $\{W_n\}_{n\in \Nb}$ is decreasing, and $\bigcap_{n\in \Nb}W_n=\emptyset$.

We define an internal set $Z=\prod_\fF \cZ_i$ as follows. If $i\in \Nb \setminus W_1$, we take any $\cZ_i \subseteq \{1, \dots, d_i\}$.
 If $i\in W_n\setminus W_{n+1}$ for some $n\in \Nb$, we take $\cZ_i$ to be a $(\rho, F_n, 1/n, \sigma_i)$-independence
set  for $\oA$ with $|\cZ_i|/d_i\ge q$. Then $\cZ_i$ is a $(\rho, F_n, 1/n, \sigma_i)$-independence set  for $\oA$ for all $n\in \Nb$
and $i\in W_n$. Thus $Z$ is an internal independence set for $\oA$.
As $|\cZ_i|/d_i\ge q$ for all $i\in W_1$, we have $\mu(Z)=\lim_{i\to \fF}\zeta(Z_i)\ge q$. This proves the ``only if'' part.

(2). Let $Z=\prod_\fF \cZ_i$ be an internal independence set for $\oA$, and let $\tau=(\tau_i)_i\in G'$. Then $\tau Z=\prod_\fF \tau_i\cZ_i$.
Let $F$ be a nonempty finite subset of $G$ and $1>\delta>0$. Then the set $W$ of all $i\in \Nb$ for which $\cZ_i$ is a
$(\rho, F, \delta, \sigma_i)$-independence set for $\oA$ is in $\fF$.
Since $\tau\in G'$, the set $V$ of all $i\in \Nb$ for which
$\max_{s\in F}\rho_{\rm Hamm}(\tau^{-1}_i \sigma_{i,s}, \sigma_{i,s}\tau^{-1}_i) \le \delta^2$ is also in $\fF$.
Then $V\cap W$ is in $\fF$. For every $i\in V$, $\varphi\in \Map(\rho, F, \delta, \sigma_i)$, and $s\in F$ one has
\begin{align*}
\lefteqn{\rho_2(\alpha_s\circ \varphi\circ \tau^{-1}_i, \varphi \circ \tau^{-1}_i \circ \sigma_{i,s})}\hspace*{20mm} \\
\hspace*{20mm} &\le \rho_2(\alpha_s\circ \varphi\circ \tau^{-1}_i, \varphi \circ \sigma_{i,s}\circ \tau^{-1}_i)
+\rho_2(\varphi\circ \sigma_{i,s} \circ \tau^{-1}_i, \varphi \circ \tau^{-1}_i \circ \sigma_{i,s}) \\
&\le  \rho_2(\alpha_s\circ \varphi, \varphi  \circ \sigma_{i,s})
+(\rho_{\rm Hamm}(\tau^{-1}_i \circ \sigma_{i,s}, \sigma_{i,s} \circ \tau^{-1}_i))^{1/2}\\
&\le \delta+\delta= 2\delta,
\end{align*}
where $\alpha_s$ is the transformation $x\mapsto sx$ of $X$, and hence $\varphi \circ \tau^{-1}_i\in \Map(\rho, F, 2\delta, \sigma_i)$. It follows that for every $i\in V\cap W$ the set $\tau_i \cZ_i$
is a $(\rho, F, 2\delta, \sigma_i)$-independence set for $\oA$. Therefore $\tau Z$ is an internal independence set for $\oA$.

(3). This can be proved using arguments similar to the proof of Proposition~\ref{P-entropy for product}.
\end{proof}

\begin{lemma}\label{L-ergodic to positive intersection}
Suppose that $\Gamma$ is a subgroup of $\fG$ and the action of $\Gamma$ on $(\prod_\fF\{1, \dots, d_i\}, \fB, \mu)$ is ergodic.
Let $Y, Z\in \fB$ be such that $\mu(Y), \mu(Z)>0$. Then $\mu(Y\cap \tau Z)>0$ for some $\tau\in \Gamma$.
\end{lemma}

\begin{proof}
Set $r=\sup_F \mu(\bigcup_{\tau \in F} \tau Z)$ with $F$ ranging over the nonempty countable subsets of $\Gamma$.
Then we can find nonempty finite subsets
$F_1, F_2, \dots$ of $\Gamma$ such that $r=\lim_{n\to \infty}\mu(\bigcup_{\tau \in F_n} \tau Z)$. Set $W=\bigcup_{n\in \Nb} F_n$
and $Z'=\bigcup_{\tau \in W} \tau Z$. Then $W$ is a countable subset of $\Gamma$ and $r=\mu(Z')$. For every $\tau'\in \Gamma$
we have $\mu(Z\cup \tau'Z)=\mu(\bigcup_{\tau \in W\cup \tau' W} \tau Z)\le r$ and hence $\mu(\tau' Z\setminus Z)=0$.
Since the action of $\Gamma$ on $(\prod_\fF\{1, \dots, d_i\}, \fB, \mu)$ is ergodic, we conclude that $r=1$.
Thus $\mu(Y)=\mu(Y\cap Z')\le \sum_{\tau \in W}\mu(Y \cap \tau Z)$, and hence $\mu(Y \cap \tau Z)>0$ for some $\tau \in W$.
\end{proof}

\begin{lemma} \label{L-tuple of product sets}
Suppose that the action of $G'$ on $(\prod_\fF\{1, \dots, d_i\}, \fB, \mu)$ is ergodic. Let $G$ act continuously on compact metrizable
spaces $X$ and $Y$. Let $k\in\Nb$ and let $\oA$ and $\oB$ be $k$-tuples of subsets of $X$ and $Y$, respectively.
Suppose that both $\oA$ and $\oB$ have positive upper independence density over $\Sigma$.
Then $\oA\times \oB$ also has positive upper independence density over $\Sigma$.
\end{lemma}

\begin{proof}
This follows from Proposition~\ref{P-basics of internal independence sets} and Lemma~\ref{L-ergodic to positive intersection}.
\end{proof}

Theorem~\ref{T-ergodic to product} now follows from Proposition~\ref{P-product for IE} and Lemma~\ref{L-tuple of product sets}.

The remainder of this section is devoted to the problem of when the ergodicity hypothesis in Theorem~\ref{T-ergodic to product} is satisfied.
We prove that this happens when $G$ is residually finite and $\Sigma$ arises from
finite quotients of $G$, and also when $G$ is amenable and $\Sigma$ is arbitrary.
A combination of results of Elek and Szabo  \cite[Thm.\ 2]{EleSza11} and Paunescu \cite{Pau11} shows on the other hand that
if $G$ is nonamenable then there is always a sofic approximation sequence $\Sigma$ for which
the commutant $G'$ does not act ergodically.

Let $G$ be an infinite residually finite group, and let $\{G_i\}_{i\in \Nb}$ be a sequence of finite-index normal subgroups
of $G$ such that $\bigcap_{n\in \Nb}\bigcup_{i\ge n}G_i=\{e\}$. Then we have the sofic approximation sequence
$\Sigma = \{ \sigma_i : G\to\Sym (|G/G_i| ) \}$ by identifying $\{1, \dots, |G/G_i|\}$ with $G/G_i$ and setting
$\sigma_i(s)(tG_i)=stG_i$ for $s, t\in G$.

\begin{theorem} \label{T-rf ergodic}
Under the above hypotheses, the action of $G'$ on $(\prod_\fF\{1, \dots, |G/G_i|\}, \fB, \mu)$ is ergodic.
\end{theorem}

\begin{proof}
Consider the right multiplication action $\sigma'$ of $G/G_i$ on itself given by $\sigma'_i(sG_i)(tG_i)=ts^{-1}G_i$ for $s, t\in G$.
Since this commutes with $\sigma_i$, it suffices to show that the action of $\prod_\fF\sigma'_i(G/G_i)\subseteq G'$
on $(\prod_\fF\{1, \dots, |G/G_i|\}, \fB, \mu)$ is ergodic.

Let $\cY_i\subseteq G/G_i$. Then, using the $\ell^1$-norm with respect to the uniform probability measure on $G/G_i$,
\begin{align*}
\frac{1}{|G/G_i|}\sum_{sG_i\in G/G_i}|\sigma'_i(sG_i)\cY_i\Delta \cY_i|
&= \sum_{sG_i\in G/G_i}\big\|1_{\sigma'_i(sG_i)\cY_i}-1_{\cY_i}\big\|_1\\
&\ge \bigg\|\sum_{sG_i\in G/G_i}1_{\sigma'_i(sG_i)\cY_i}-|G/G_i|\cdot 1_{\cY_i} \bigg\|_1\\
&= \big\||\cY_i|\cdot 1_{G/G_i}-|G/G_i|\cdot 1_{\cY_i}\big\|_1=2\frac{|\cY_i|}{|G/G_i|}\bigg(|G/G_i|-|\cY_i|\bigg).
\end{align*}
Thus there is some $s_iG_i\in G/G_i$ with
$\frac{1}{|G/G_i|}|\sigma'_i(s_iG_i)\cY_i\Delta \cY_i|\ge 2\frac{|\cY_i|}{|G/G_i|}\big(1-\frac{|\cY_i|}{|G/G_i|}\big)$.

Let $Y=\prod_\fF \cY_i$ be an internal subset of $\prod_\fF\{1, \dots, |G/G_n|\}$.  Take $s_iG_i \in G/G_i$ as above for each $i\in \Nb$.
Set $s=(s_iG_i)_i$.
Then
\begin{align*}
\mu(\sigma'(s)Y\Delta Y)&=\lim_{n\to \fF}\frac{|\sigma'(s_iG_i)\cY_i\Delta \cY_i|}{|G/G_i|} \\
&\ge \lim_{n\to \fF}2\frac{|\cY_i|}{|G/G_i|}\bigg(1-\frac{|\cY_i|}{|G/G_i|}\bigg)=2\mu(Y)(1-\mu(Y)).
\end{align*}
If $\mu(\sigma'(s)Y\Delta Y)=0$, then $\mu(Y)=0$ or $1$. This finishes the proof.
\end{proof}

\begin{theorem}\label{T-amenable ergodic}
Let $G$ be a countable amenable group. For every sofic approximation sequence $\Sigma$ for $G$,
the action of $G'$ on $(\prod_\fF\{1, \dots, d_i\}, \fB, \mu)$ is ergodic.
\end{theorem}

The proof of Theorem~\ref{T-amenable ergodic} requires several lemmas.

We will use the following terminology.
Let $(X,\mu )$ be a finite measure space and let $\delta \geq 0$.
A family of measurable subsets of $X$ is said to {\it $\delta$-cover}
$X$ if its union has measure at least $\delta \mu (X)$.
A collection $\{ A_i \}_{i\in I}$ of positive measure sets is
{\it $\delta$-disjoint} if there exist pairwise disjoint sets $\widehat{A}_i \subseteq A_i$ such that
$\mu (\widehat{A}_i ) \geq (1-\delta ) \mu (A_i )$ for all $i\in I$.

The following is the Rokhlin lemma for sofic approximations, which is based on the quasitiling theory
of Ornstein and Weiss and appears as Lemma~4.5 in \cite{KerLi10}. The statement of the latter does not contain
condition (3) below, but it is not hard to see from the proof in \cite{KerLi10} that it can be arranged.

\begin{lemma}\label{L-Rokhlin1}
Let $G$ be a countable discrete group.
Let $0\le \theta<1$, and $0<\eta<1$. Then there are an $\ell'\in \Nb$ and $\kappa, \eta''>0$
such that, whenever $e\in F_1\subseteq F_2\subseteq \cdots \subseteq F_{\ell'}$ are finite subsets of $G$
with $|(F_{k-1}^{-1}F_k) \setminus F_k|\le \kappa |F_k|$ for $k=2, \dots, \ell'$, there exist
$\lambda_1 , \dots, \lambda_{\ell'} \in [0,1]$
such that
for every $\delta>0$, every sufficiently large $d\in \Nb$ (depending on $\delta$),
every map $\sigma: G\rightarrow \Sym(d)$ with a set $\cB\subseteq \{1, \dots, d\}$ satisfying $|\cB|\ge (1-\eta'')d$ and
\[
\sigma_{st}(a)=\sigma_s\sigma_t(a), \ \sigma_s(a)\neq \sigma_{s'}(a),\ \sigma_e(a)=a
\]
for all $a\in \cB$ and $s, t, s'\in F_{\ell'}\cup F_{\ell'}^{-1}$ with $s\neq s'$, and every set $\cV\subseteq \{1, \dots, d\}$ with $|\cV|\ge (1-\theta)d$,
there exist $\cC_1, \dots, \cC_{\ell'}\subseteq \cV$ such that
\begin{enumerate}
\item for every $k=1, \dots, \ell'$ and $c\in \cC_k$, the map $s\mapsto \sigma_s(c)$ from $F_k$ to $\sigma(F_k)c$ is bijective,

\item the sets $\sigma(F_1)\cC_1, \dots, \sigma(F_{\ell'})\cC_{\ell'}$ are pairwise disjoint
and the family $\bigcup_{k=1}^{\ell'}\{\sigma(F_k)c: c\in \cC_k\}$ is $\eta$-disjoint and $(1-\theta-\eta)$-covers $\{1, \dots, d\}$,

\item $\sum_{k=1}^{\ell'}||\sigma(F_k)\cC_k|/d-\lambda_k|<\delta$.
\end{enumerate}
\end{lemma}

\begin{lemma}\label{L-density}
Let $G$ be a countable discrete group. Let
$F$ be a nonempty finite subset of $G$. For every $d\in \Nb$, every map $\sigma: G\rightarrow \Sym(d)$,
every set $\cB\subseteq \{1, \dots, d\}$ satisfying
\[
\sigma_s(a)\neq \sigma_t(a)
\]
for all $a\in \cB$ and distinct $s, t\in F$, every $\cJ\subseteq \{1, \dots, d\}$, and every $0<\lambda<1$,
there exists a $\cV\subseteq \cB$ such that $|\cV|\ge \frac{|\cB|(1-\lambda)-d+|\cJ| }{1-\lambda}$ and
$|\sigma(F)a\cap \cJ|>\lambda |F| $ for all $a\in \cV$.
\end{lemma}

\begin{proof}
Set $X=\{1, \dots, d\}$.
Denote by $\zeta$ the uniform probability measure on $X$.
One has
\begin{align*}
\frac{1}{d}\sum_{a\in \cB}|\sigma(F)a\cap \cJ|&=\int_{\cJ}\sum_{a\in \cB} \mathbf{1}_{\sigma(F)a}(x)\, d\zeta(x)\\
&=\int_X\sum_{a\in \cB}\mathbf{1}_{\sigma(F)a}(x)\, d\zeta(x)-\int_{X\setminus \cJ}\sum_{a\in \cB} \mathbf{1}_{\sigma(F)a}(x)\, d\zeta(x)\\
&\ge \frac{|\cB|\cdot |F|}{d}-\int_{X\setminus \cJ}|F| \, d\zeta(x)\\
&=\frac{|\cB|\cdot |F|}{d}-\bigg(1-\frac{|\cJ|}{d}\bigg)|F|.
\end{align*}
Set $\cV=\{a\in \cB: |\sigma(F)a\cap \cJ|> \lambda |F|\}$. Then
\begin{align*}
\frac{1}{d}\sum_{a\in \cB}|\sigma(F)a\cap \cJ|\le \frac{|\cV|\cdot |F|}{d}+\frac{(|\cB|-|\cV|)\lambda |F|}{d}.
\end{align*}
Thus
\[
\frac{|\cV|\cdot |F|}{d}+\frac{(|\cB|-|\cV|)\lambda |F|}{d}\ge \frac{|\cB|\cdot |F|}{d}-\bigg(1-\frac{|\cJ|}{d}\bigg)|F|.
\]
It follows that
\[
|\cV|\ge \frac{|\cB|(1-\lambda)-d+|\cJ| }{1-\lambda}.
\qedhere
\]
\end{proof}

The proof of Lemma~4.4 in \cite{KerLi10} shows the following.

\begin{lemma}\label{L-disjointcover}
Let $(X,\mu )$ be a finite measure space.
Let $\delta ,\eta\in [0,1)$ and
let $\{ A_i \}_{i\in I}$ be a finite $\delta$-even covering of $X$ by positive measure sets.
Then every $\eta$-disjoint subcollection of $\{ A_i \}_{i\in I}$ can be enlarged to an $\eta$-disjoint
subcollection of $\{ A_i \}_{i\in I}$ which $\eta (1-\delta )$-covers $X$.
\end{lemma}

\begin{lemma}\label{L-bijection}
Let $G$ be a countable discrete group.
Let $F$ be a nonempty finite subset of $G$, $0<\tau\le 1$, and $0<\eta<1/2$. Then
for every large enough $d\in \Nb$, every map $\sigma: G\rightarrow \Sym(d)$
with sets $\cB_1, \cB_2 \subseteq \{1, \dots, d\}$ satisfying
$|\cB_i|\ge (\frac{\tau}{2}+\frac{2-2\tau}{2-\tau})d$ and
\[\sigma_s(a)\neq \sigma_{t}(a)\]
for all $a\in \cB_i$ and distinct $s, t\in F$,
and every $\cJ_1, \cJ_2\subseteq \{1, \dots, d\}$ with $|\cJ_i|\ge \tau d$ for $i=1, 2$,
there exist $\cC_i\subseteq \cB_i$ such that
\begin{enumerate}
\item for every $i=1, 2$, the family $\{\sigma(F)c: c\in \cC_i\}$ is $\eta$-disjoint
and $\eta\frac{\tau}{16}$-covers $\{1, \dots, d\}$,

\item there is a bijection $\varphi: \cC_1\rightarrow \cC_2$ such that for any $c\in \cC_1$, one has
$|\{s\in F: \sigma_s(c)\in \cJ_1, \sigma_s(\varphi(c))\in \cJ_2\}|\ge (\frac{\tau}{2})^2 |F|$.
\end{enumerate}
\end{lemma}

\begin{proof}
Note that
for all distinct $a, c\in \{1, \dots, d\}$ and $s\in F$ we have
\[
\sigma_s(a)\neq \sigma_s(c).
\]

Taking $\lambda=\tau/2$ and $\cJ=\cJ_1$ in Lemma~\ref{L-density}, we find a $\cV_1\subseteq \cB_1$
such that $|\cV_1|/d\ge \frac{|\cB_1|(1-\lambda)/d-1+|\cJ_1|/d}{1-\lambda}\ge \frac{(\tau/2+(2-2\tau)/(2-\tau))(1-\lambda)-1+\tau}{1-\lambda}=\frac{\tau}{2}$
and $|\sigma(F)a\cap \cJ_1|/|F|\ge \frac{\tau}{2}$ for all $a\in \cV_1$.
Observe that
\[
\sum_{c\in \cV_1}|\sigma(F)c|=|F|\cdot |\cV_1|\ge |F|\cdot \frac{\tau}{2}d=|F|\cdot \bigg(1-\frac{2-\tau}{2}\bigg)d ,
\]
so that the family $\{\sigma(F)c\}_{c\in \cV_1}$ is a $\frac{2-\tau}{2}$-even covering of $\{1, \dots, d\}$ with
multiplicity $|F|$. By Lemma~\ref{L-disjointcover}, we can find a set
$\cW_1\subseteq \cV_1$ such that the family $\{\sigma(F)c\}_{c\in \cW_1}$ is $\eta$-disjoint and
$\eta\frac{\tau}{2}$-covers $\{1, \dots, d\}$.
We may assume that $|\sigma(F)\cW_1|< \eta \tau d/2 +|F|$.

List all the subsets of $F$ with cardinality $\lceil |F|\tau/2\rceil$ as  $F_1, \dots, F_n$ for some $n\in \Nb$,
where $\lceil x\rceil$ for a real number $x$ denotes the smallest integer no less than $x$.
Then we can write $\cW_1$ as the disjoint union of sets $\cW_{1, j}$ for $1\le j\le n$ such that $\sigma(F_j)c\subseteq \cJ_1$
for all $1\le j\le n$ and $c\in \cW_{1, j}$.  Throwing away those empty
$\cW_{1, j}$, we may assume that each $\cW_{1, j}$ is nonempty.

For each $1\le j\le n$, taking $\lambda=\tau/2$ and $\cJ=\cJ_2$ in Lemma~\ref{L-density}, we  find a
$\cV_{2, j}\subseteq \cB_2$ such that
$|\cV_{2, j}|/d\ge \frac{|\cB_2|(1-\lambda)/d-1+|\cJ_2|/d}{1-\lambda}\ge \frac{(\tau/2+(2-2\tau)/(2-\tau))(1-\lambda)-1+\tau}{1-\lambda}=\frac{\tau}{2}$
and $|\sigma(F_j)a\cap \cJ_2|/|F_j|\ge \frac{\tau}{2}$ for all $a\in \cV_{2, j}$.

We will recursively construct pairwise disjoint sets $\cC_{2, 1}, \dots, \cC_{2, n}$ such that the family
 $\{\sigma(F)c: c\in \cC_{2, j}, 1\le j\le n\}$ is $\eta$-disjoint, and $\cC_{2, j}\subseteq \cV_{2, j}$ and $|\cC_{2,  j}|=\lfloor |\cW_{1,  j}|/2\rfloor$
for every $1\le j\le n$, where $\lfloor x\rfloor$  for any real number $x$ denotes the largest integer no bigger than $x$.

Note that
\[
\sum_{c\in \cV_{2,  1}}|\sigma(F)c|=|F|\cdot |\cV_{2,  1}|\ge |F|\cdot \frac{\tau}{2}d=|F|\cdot \bigg(1-\frac{2-\tau}{2}\bigg)d ,
\]
so that the family $\{\sigma(F)c\}_{c\in \cV_{2, 1}}$ is a $\frac{2-\tau}{2}$-even covering of $\{1, \dots, d\}$ with
multiplicity $|F|$. By Lemma~\ref{L-disjointcover}, we can find a set
$\cW_{2,  1}\subseteq \cV_{2,  1}$ such that the family $\{\sigma(F)c\}_{c\in \cW_{2,  1}}$ is $\eta$-disjoint and
$\eta\frac{\tau}{2}$-covers $\{1, \dots, d\}$. Note that
\[
|\cW_{2, 1}|\cdot |F|\ge |\sigma(F)\cW_{2,  1}|\ge \eta\frac{\tau}{2}d,
\]
and since the family $\{\sigma(F)c\}_{c\in \cW_1}$ is $\eta$-disjoint, we have
\begin{align} \label{E-Rokhlin1}
 \frac{1}{2}|\cW_1|\cdot |F|\le(1-\eta)|\cW_1|\cdot |F|\le |\sigma(F)\cW_1|< \eta\frac{\tau}{2}d+|F|.
\end{align}
Thus
\[
\frac{1}{2}|\cW_1|\cdot |F|<|\cW_{2, 1}|\cdot |F|+|F|,
\]
and hence
\[
|\cW_{2, 1}|> \frac{1}{2}|\cW_1|-1\ge \frac{1}{2}|\cW_{1, 1}|-1.
\]
Therefore we can take a subset $\cC_{2, 1}$ of $\cW_{2, 1}$ with cardinality $\lfloor \frac{1}{2}|\cW_{1,  1}|\rfloor$.

Suppose that we have found pairwise disjoint sets $\cC_{2,  1}, \dots, \cC_{2,  k}$ for some $1\le k<n$
such that the family $\{\sigma(F)c: c\in \cC_{2, j}, 1\le j\le k\}$ is $\eta$-disjoint, and $\cC_{2,  j}\subseteq \cV_{2,  j}$
and $|\cC_{2, j}|=\lfloor |\cW_{1,  j}|/2\rfloor$ for every $1\le j\le k$. Note that
\begin{align*}
\sum_{c\in \cV_{2,  k+1}\cup \bigcup_{1\le j\le k}\cC_{2, j}}|\sigma(F)c|
&=|F |\cdot \bigg|\cV_{2, k+1}\cup \bigcup_{1\le j\le k}\cC_{2, j} \bigg| \\
&\ge |F|\cdot |\cV_{2,  k+1}|\ge |F|\cdot \frac{\tau}{2}d=|F|\cdot \bigg(1-\frac{2-\tau}{2}\bigg)d ,
\end{align*}
so that the family $\{\sigma(F)c\}_{c\in \cV_{2,  k+1}\cup \bigcup_{1\le j\le k}\cC_{2,  j}}$ is a $\frac{2-\tau}{2}$-even
covering of $\{1, \dots, d\}$ with multiplicity $|F|$. By Lemma~\ref{L-disjointcover}, we can find a set
$\cW_{2, k+1}\subseteq \cV_{2,  k+1}\setminus \bigcup_{1\le j\le k}\cC_{2, j}$ such that the family
$\{\sigma(F)c\}_{c\in \cW_{2,  k+1}\cup \bigcup_{1\le j\le k}\cC_{2,  j}}$ is $\eta$-disjoint and
$\eta\frac{\tau}{2}$-covers $\{1, \dots, d\}$. Note that
\[
\bigg(|\cW_{2, k+1}|+\sum_{1\le j\le k}|\cC_{2,  j}|\bigg)\cdot |F|
\ge \bigg|\sigma(F)\bigg(\cW_{2, k+1}\cup \bigcup_{1\le j\le k}\cC_{2,  j}\bigg)\bigg|\ge \eta\frac{\tau}{2}d.
\]
Thus, combining with \eqref{E-Rokhlin1}, we have
\begin{align*}
\frac{1}{2}\sum_{1\le j\le k+1}|\cW_{1, j}|\cdot |F|&\le \frac{1}{2}|\cW_1|\cdot |F|\\
&<\bigg(|\cW_{2, k+1}|+\sum_{1\le j\le k}|\cC_{2, j}|\bigg)\cdot |F|+|F|\\
&\le \bigg(|\cW_{2, k+1}|+\frac{1}{2}\sum_{1\le j\le k}|\cW_{1, j}|\bigg)\cdot |F|+|F|,
\end{align*}
and hence
\[
|\cW_{2, k+1}|> \frac{1}{2}|\cW_{1, k+1}|-1.
\]
Therefore we can take a subset $\cC_{2, k+1}$ of $\cW_{2, k+1}$ with cardinality $\lfloor \frac{1}{2}|\cW_{1, k+1}|\rfloor$.
This completes the recursive construction.

For each $1\le j\le n$ take a subset $\cC_{1, j}$ of $\cW_{1, j}$ with cardinality $\lfloor |\cW_{1, j}|/2\rfloor$.
Set $\cC_i=\bigcup_{1\le j\le n}\cC_{i, j}$ for $i=1, 2$. Take a bijection $\varphi: \cC_1\rightarrow \cC_2$ such that
$\varphi(\cC_{1, j})=\cC_{2, j}$ for all $1\le j\le n$. For each $c\in \cC_1$, considering $j$ such that $c\in \cC_{1, j}$ one has
\begin{align*}
|\{s\in F: \sigma_s(c)\in \cJ_1, \sigma_s(\varphi(c))\in \cJ_2\}|
&\ge |\{s\in F_j: \sigma_s(\varphi(c))\in \cJ_2\}| \\
&\ge \frac{\tau}{2}|F_j| \ge \bigg(\frac{\tau}{2}\bigg)^2|F|.
\end{align*}
Note that
\[
|\cW_1|\cdot |F|\ge |\sigma(F)\cW_1|\ge \eta\frac{\tau}{2}d,
\]
and hence for $i=1, 2$,
\[
|\cC_i|=\sum_{1\le j\le n}| \cC_{i, j}|\ge \sum_{1\le j\le n}\bigg(\frac{1}{2}|\cW_{1,  j}|-1\bigg)
\ge \frac{1}{2}|\cW_1|-2^{|F|}\ge \frac{1}{4}|\cW_1|
\]
when $d$ is sufficiently large.
Since the family $\{\sigma(F)c: c\in \cC_i\}$ is $\eta$-disjoint, we get
\[
|\sigma(F)\cC_i|\ge (1-\eta)|\cC_i|\cdot |F|\ge (1-\eta)\eta\frac{\tau}{8}d\ge \eta\frac{\tau}{16}d.
\]
\end{proof}

\begin{lemma} \label{L-Rokhlin}
Let $G$ be a countable discrete group.
Let $0<\tau\le 1$, and $0<\eta<1/2$ with $\eta\frac{\tau}{16}<\frac{1-\tau'}{24}$,
where $\tau'=\frac{\tau}{2}+\frac{2-2\tau}{2-\tau}<1$.
Then there are an $\ell\in \Nb$ and $\eta'>0$
such that, whenever $e\in F_1\subseteq F_2\subseteq \cdots \subseteq F_\ell$
are finite subsets of $G$ with $|(F_{k-1}^{-1}F_k) \setminus F_k|\le |F_k|$
for $k=2, \dots, \ell$,
for every large enough $d\in \Nb$, every map $\sigma: G\rightarrow \Sym(d)$ with a set $\cB\subseteq \{1, \dots, d\}$
satisfying $|\cB|\ge (1-\eta')d$ and
\[\sigma_{st}(a)=\sigma_s\sigma_t(a),\ \sigma_s(a)\neq \sigma_{s'}(a),\ \sigma_e(a)=a\]
for all $a\in \cB$ and $s, t, s'\in F_\ell\cup F_\ell^{-1}$ with $s\neq s'$,
and any $\cJ_1, \cJ_2\subseteq \{1, \dots, d\}$ with $|\cJ_i|\ge \tau d$ for $i=1, 2$,
there exist $\cC_{i, 1}, \dots, \cC_{i, \ell}\subseteq \cB$ such that
\begin{enumerate}
\item for every $i=1, 2$, $k=1, \dots, \ell$ and $c\in \cC_{i, k}$, the map $s\mapsto \sigma_s(c)$
from $F_k$ to $\sigma(F_k)c$ is bijective,

\item for every $i=1, 2$, the sets $\sigma(F_1)\cC_{i, 1}, \dots, \sigma(F_\ell)\cC_{i, \ell}$ are pairwise disjoint,
 the family $\bigcup_{k=1}^\ell\{\sigma(F_k)c: c\in \cC_{i, k}\}$ is $\eta$-disjoint, and
$(1-\eta)\frac{1-\tau'}{24}d\le |\bigcup_{k=1}^\ell\sigma(F_k)\cC_{i, k}|\le (\frac{1}{1-\eta}\frac{1-\tau'}{24}+\eta)d$,

\item for every $k=1, \dots, \ell$, there is a bijection $\varphi_k: \cC_{1, k}\rightarrow \cC_{2, k}$
such that for each $c\in \cC_{1, k}$ one has
$|\{s\in F_k: \sigma_s(c)\in \cJ_1, \sigma_s(\varphi_k(c))\in \cJ_2\}|\ge (\frac{\tau}{2})^2 |F_k|$.
\end{enumerate}
\end{lemma}

\begin{proof}
Set $\eta'=\frac{1-\tau'}{2}$. Take $\ell$ to be the largest integer satisfying
$\ell \eta\frac{\tau}{16}\le \frac{1-\tau'}{12}$. Then $\ell \eta\frac{\tau}{16}\ge \frac{1-\tau'}{24}$.
We will recursively construct  sets $\cC_{i, 1}', \dots, \cC_{i, \ell}'$ in reverse order
so that (i) for every $i=1, 2$ and $1\le k\le \ell$, the sets $\sigma(F_k)\cC_{i, k}', \dots, \sigma(F_\ell)\cC_{i, \ell}'$
are pairwise disjoint and the family $\bigcup_{n=k}^\ell\{\sigma(F_n)c: c\in \cC'_{i, n}\}$ is $\eta$-disjoint
and $(\ell-k+1)\eta\frac{\tau}{16}$-covers $\{1, \dots, d\}$, and (ii)
for every $k=1, \dots, \ell$, there is a bijection $\varphi_k: \cC_{1, k}'\rightarrow \cC_{2, k}'$ such that
for every $c\in \cC_{1, k}'$ one has
$|\{s\in F_k: \sigma_s(c)\in \cJ_1, \sigma_s(\varphi_k(c))\in \cJ_2\}|\ge (\frac{\tau}{2})^2 |F_k|$.

Taking $\cB_i=\cB$ for $i=1, 2$ in Lemma~\ref{L-bijection} we  find $\cC_{i, \ell}'\subseteq \cB$ for $i=1, 2$
such that the family $\{\sigma(F_\ell)c: c\in \cC_{i, \ell}'\}$ is $\eta$-disjoint and $\eta\frac{\tau}{16}$-covers
$\{1, \dots, d\}$ for $i=1, 2$ and there is a bijection $\varphi_\ell: \cC_{1, \ell}'\rightarrow \cC_{2, \ell}'$ with
$|\{s\in F_\ell: \sigma_s(c)\in \cJ_1, \sigma_s(\varphi_\ell(c))\in \cJ_2\}|\ge (\frac{\tau}{2})^2|F_\ell|$
for all $c\in \cC_{1, \ell}'$.

Suppose that $1\le k<\ell$ and we have found $\cC_{i, k+1}', \dots, \cC_{i, \ell}'\subseteq \cB$ for $i=1, 2$ such that the sets
$\sigma(F_{k+1})\cC_{i, k+1}', \dots, \sigma(F_\ell)\cC_{i,\ell}'$ are pairwise disjoint and
the family $\bigcup_{n=k+1}^\ell\{\sigma(F_n)c: c\in \cC_{i, n}'\}$ is $\eta$-disjoint and $(\ell-k)\eta\frac{\tau}{16}$-covers
$\{1, \dots, d\}$ for each $i=1,2$, and there is a bijection $\varphi_j:\cC_{1, j}'\rightarrow \cC_{2, j}'$
with $|\{s\in F_j: \sigma_s(c)\in \cJ_1, \sigma_s(\varphi_j(c))\in \cJ_2\}|\ge (\frac{\tau}{2})^2|F_j|$ for all $j=k+1, \dots, \ell$ and $c\in \cC_{1, j}'$.
Set $\theta_{i, k}=|\bigcup_{j=k+1}^\ell \sigma(F_j)\cC_{i, j}'|/d$ and
$\cB_{i, k}=\big\{c\in \cB: \sigma(F_k)c\cap \big(\bigcup_{j=k+1}^\ell\sigma(F_j)\cC_{i, j}'\big)=\emptyset\big\}$ for $i=1, 2$.

If $1-\tau'-\eta'< 3\theta_{m, k}$ for some $m=1, 2$, then
we set $\cC_{i, k}'=\emptyset$ for each $i=1, 2$.
Then
\begin{align*}
\sum_{j=k}^\ell|\cC'_{m, j}|\cdot |F_j|\ge \bigg|\bigcup_{j=k}^\ell\sigma(F_j)\cC_{m, j}' \bigg|
=\theta_{m, k}d\ge \frac{1-\tau'-\eta'}{3}d= \frac{1-\tau'}{6}d.
\end{align*}
Since the family $\bigcup_{j=k}^\ell\{\sigma(F_j)c: c\in \cC_{i, j}'\}$ is $\eta$-disjoint, one has
\begin{align*}
\bigg|\bigcup_{j=k}^\ell\sigma(F_j)\cC_{i, j}' \bigg|
&\ge (1-\eta)\sum_{j=k}^\ell|\cC_{i, j}'|\cdot |F_j| \\
&\ge (1-\eta)\frac{1-\tau'}{6}d\ge \frac{1-\tau'}{12}d\ge \ell\eta\frac{\tau}{16}d\ge (\ell-k+1)\eta\frac{\tau}{16}d
\end{align*}
for $i=1, 2$.

Assume that $1-\tau'-\eta' \ge 3\theta_{i, k}$ for every $i=1, 2$. Let $i\in \{1, 2\}$.
For every $c\in \cB\setminus  \cB_{i, k}$ we have $\sigma_s(c)=\sigma_t(a)$ for some
$j\in \{k+1, \dots, \ell \}$, $a\in \cC_{i, j}'$, $t\in F_j$, and $s\in F_k$, and hence
\[
c=\sigma_{s^{-1}}\sigma_s(c)=\sigma_{s^{-1}}\sigma_t(a)=\sigma_{s^{-1}t}(a)\in \bigcup_{j=k+1}^\ell\sigma(F_k^{-1}F_j)\cC_{i, j}'.
\]
Therefore
\[
\cB\setminus \cB_{i, k}\subseteq \bigcup_{j=k+1}^\ell\sigma(F_k^{-1}F_j)\cC_{i, j}'.
\]
Since the family $\bigcup_{j=k+1}^\ell \{\sigma(F_j)c: c\in \cC_{i, j}'\}$ is $\eta$-disjoint we have
\[
\frac{1}{2}\sum_{j=k+1}^\ell |F_j|\cdot |\cC_{i, j}'|\le \sum_{j=k+1}^\ell (1-\eta)|F_j|\cdot |\cC_{i, j}'|
\le \bigg|\bigcup_{j=k+1}^\ell \sigma(F_j)\cC_{i,j}' \bigg|=\theta_{i, k}d.
\]
Thus
\begin{align*}
\bigg|\bigcup_{j=k+1}^\ell \sigma(F_k^{-1}F_j)\cC_{i, j}'\bigg|
&\le \bigg|\bigcup_{j=k+1}^\ell \sigma((F_k^{-1}F_j)\setminus F_j)\cC_{i, j}'\bigg|+\bigg|\bigcup_{j=k+1}^\ell \sigma(F_j)\cC_{i, j}'\bigg|\\
&\le \sum_{j=k+1}^\ell |(F_k^{-1}F_j)\setminus F_j|\cdot |\cC_{i, j}'|+\theta_{i, k}d\\
&\le \sum_{j=k+1}^\ell |(F_{j-1}^{-1}F_j)\setminus F_j|\cdot |\cC_{i, j}'|+\theta_{i, k}d\\
&\le \sum_{j=k+1}^\ell |F_j|\cdot |\cC_{i, j}'|+\theta_{i, k}d\\
&\le 3 \theta_{i, k}d.
\end{align*}
Therefore
\begin{align*}
|\cB_{i, k}|=|\cB|-|\cB\setminus \cB_{i, k}| &\ge (1-\eta')d-\bigg|\bigcup_{j=k+1}^\ell \sigma(F_k^{-1}F_j)\cC_{i, j}'\bigg|\\
&\ge (1-\eta')d-3\theta_{i, k}d \\
&\ge \tau' d.
\end{align*}

Taking $\cB_i=\cB_{i, k}$ in Lemma~\ref{L-bijection}, we find $\cC'_{i, k}\subseteq \cB_{i, k}$ for $i=1, 2$
such that the family $\{\sigma(F_k)c: c\in \cC_{i, k}'\}$ is $\eta$-disjoint and $\eta\frac{\tau}{16}$-covers
$\{1, \dots, d\}$ for $i=1, 2$ and there is a bijection $\varphi_k: \cC_{1, k}'\rightarrow \cC_{2, k}'$ with
$|\{s\in F_k: \sigma_s(c)\in \cJ_1, \sigma_s(\varphi_k(c))\in \cJ_2\}|\ge (\frac{\tau}{2})^2|F_k|$ for all $c\in \cC_{1, k}'$.
Then for each $i=1, 2$,
the sets $\sigma(F_k)\cC_{i, k}', \dots, \sigma(F_\ell)\cC_{i, \ell}'$ are pairwise disjoint,
the family $\bigcup_{j=k}^\ell\{\sigma(F_j)c: c\in \cC_{i, j}'\}$
is $\eta$-disjoint, and
\begin{align*}
\bigg|\bigcup_{j=k}^\ell\sigma(F_j)\cC_{i, j}'\bigg|&=|\sigma(F_k)\cC_{i, k}'|+\bigg|\bigcup_{j=k+1}^\ell\sigma(F_j)\cC_{i, j}'\bigg|\\
&\ge \eta\frac{\tau}{16}d+(\ell-k)\eta\frac{\tau}{16}d=(\ell-k+1)\eta\frac{\tau}{16}d,
\end{align*}
completing the recursive construction.

When $d$ is large enough, take a subset $\cC_{1, k}$ of $\cC_{1, k}'$ for each $1\le k\le \ell$ such that
\[
\frac{1-\tau'}{24}d\le \bigg|\bigcup_{k=1}^\ell\sigma(F_k)\cC_{1, k}\bigg|\le \frac{1-\tau'}{24}d+|F_\ell|\le \frac{1-\tau'}{24}d+\eta(1-\eta)d.
\]
Set $\cC_{2, k}=\varphi_k(\cC_{1, k})$ for each $1\le k\le \ell$.
Since the families $\bigcup_{k=1}^\ell\{\sigma(F_k)c: c\in \cC_{i, k}\}$ for $i=1, 2$ are $\eta$-disjoint, we have
\begin{align*}
\bigg|\bigcup_{k=1}^\ell \sigma(F_k)\cC_{2, k}\bigg|
\ge (1-\eta)\bigg|\bigcup_{k=1}^\ell \sigma(F_k)\cC_{1, k}\bigg|\ge (1-\eta)\frac{1-\tau'}{24}d,
\end{align*}
and
\begin{align*}
\bigg|\bigcup_{k=1}^\ell \sigma(F_k)\cC_{2, k}\bigg|
\le \frac{1}{1-\eta}\bigg|\bigcup_{k=1}^\ell \sigma(F_k)\cC_{1, k}\bigg|\le \bigg(\frac{1}{1-\eta}\cdot\frac{1-\tau'}{24}+\eta\bigg)d.
\end{align*}
\end{proof}

We are ready to prove Theorem~\ref{T-amenable ergodic}.

\begin{proof}[Proof of Theorem~\ref{T-amenable ergodic}]
It suffices to show that for any internal sets $Y=\prod_{\fF}\cY_n$ and $Z=\prod_{\fF}\cZ_n$ with strictly positive measure,
there is some $s\in G'$ with $\mu(Z\cap sY)>0$. In turn it is enough to show that there is some $\lambda>0$ such that
for every finite subset $F$ of $G$ and $\eps>0$ the set of all $n\in \Nb$ for which there is some $\varphi\in \Sym(d_n)$
satisfying $\rho_{\rm Hamm}(\varphi\sigma_s, \sigma_s\varphi)<\eps$ for all $s\in F$
and $\frac{|\varphi(\cY_n)\cap \cZ_n|}{d_n}\ge \lambda$ belongs to $\fF$.

Set $\tau=\min(\mu(Y), \mu(Z))/2$, $\tau'=\frac{\tau}{2}+\frac{2-2\tau}{2-\tau}$, and $\lambda=\frac{\tau^2(1-\tau')}{384}$.
Take $0<\eta<1/2$ to be a small number with $\eta\frac{\tau}{16}<\frac{1-\tau'}{24}$, to be determined in a moment.
Then the set $V$ of all $n\in \Nb$ satisfying $\min(|\cY_n|/d_n, |\cZ_n|/d_n)\ge \tau$ belongs to $\fF$. Let $\ell$ and $\eta'$ be as in Lemma~\ref{L-Rokhlin}.
We may assume that $\eta'<\varepsilon/2$.
Set $\theta=\frac{1}{(1-\eta)^2}\frac{1-\tau'}{24}+\frac{\eta}{1-\eta}+\eta'$. Let $\ell', \kappa$, and $\eta''$ be as  in Lemma~\ref{L-Rokhlin1}.
Take $F'_1\subseteq F'_2\subseteq \dots \subseteq F'_{\ell'}\subseteq F_1\subseteq F_2\subseteq \dots \subseteq F_\ell$
to be finite subsets of $G$ containing $e$ such that
\begin{enumerate}
\item $|((F'_{\ell'})^{-1}F_k)\setminus F_k|<\eta |F_k|$ for all $k=1, \dots, \ell$,

\item $|((F'_{k-1})^{-1}F'_k)\setminus F'_k|<\kappa |F'_k|$ for all $k=2, \dots, \ell'$ and $|\tilde{F}'_k|\ge (1-\eta)|F'_k|$
for all $k=1, \dots, \ell'$, where $\tilde{F}'_k=\{s\in F'_k: Fs\subseteq F'_k\}$, and

\item $|(F_{k-1}^{-1}F_k)\setminus F_k|<|F_k|$ for all $k=2, \dots, \ell$ and $|\tilde{F}_k|\ge (1-\eta)|F_k|$ for all $k=1, \dots, \ell$,
where $\tilde{F}_k=\{s\in F_k: Fs\subseteq F_k\}$.
\end{enumerate}
Then we have $\lambda_1, \dots, \lambda_{\ell'}$ as in Lemma~\ref{L-Rokhlin1}.
When $n\in V$ is large enough, one has $|\cB|\ge (1-\min(\eta', \eta''))d_n$ where
$\cB$ denotes the set of all $a\in \{1, \dots, d_n\}$ satisfying
\[
\sigma_{n,st}(a)=\sigma_{n, s}\sigma_{n, t}(a), \ \sigma_{n, s}(a)\neq \sigma_{n, s'}(a),\ \sigma_{n, e}(a)
\]
for all $s,t\in (F\cup F_\ell)\cup(F\cup F_\ell)^{-1}$ and distinct $s, s'\in F_\ell \cup F_\ell^{-1}$, and one has $\cC_{i, 1}, \dots, \cC_{i, \ell}\subseteq \cB$
for $i=1, 2$ as in Lemma~\ref{L-Rokhlin} for $d=d_n$,  $\sigma=\sigma_n$, $\cJ_1=\cY_n$, and $\cJ_2=\cZ_n$.

Let $i\in \{1, 2\}$. Set $\cV_i=\{c\in \cB: \sigma_n(F'_{\ell'})c \cap \bigcup_{k=1}^\ell\sigma_n(F_k)\cC_{i, k}=\emptyset\}$.
Then $\cB\setminus \cV_i\subseteq \bigcup_{k=1}^\ell \sigma_n((F'_{\ell'})^{-1}F_k)\cC_{i, k}$.
Since the family $\bigcup_{k=1}^\ell\{\sigma_n(F_k)c: c\in \cC_{i, k}\}$  is $\eta$-disjoint, one has
 \begin{align*}
|\cB\setminus \cV_i|&\le \bigg|\bigcup_{k=1}^\ell \sigma_n((F'_{\ell'})^{-1}F_k)\cC_{i, k}\bigg|\\
&\le
\bigg|\bigcup_{k=1}^\ell \sigma_n(((F'_{\ell'})^{-1}F_k)\setminus F_k)\cC_{i, k}\bigg|+\bigg|\bigcup_{k=1}^\ell \sigma_n(F_k)\cC_{i, k}\bigg|\\
&\le \sum_{k=1}^\ell |\sigma_n(((F'_{\ell'})^{-1}F_k)\setminus F_k)\cC_{i, k}|+\bigg|\bigcup_{k=1}^\ell \sigma_n(F_k)\cC_{i, k}\bigg|\\
&\le \sum_{k=1}^\ell |((F'_{\ell'})^{-1}F_k)\setminus F_k|\cdot |\cC_{i, k}|+\bigg|\bigcup_{k=1}^\ell \sigma_n(F_k)\cC_{i, k}\bigg|\\
&\le \eta\sum_{k=1}^\ell |F_k|\cdot |\cC_{i, k}|+\bigg|\bigcup_{k=1}^\ell \sigma_n(F_k)\cC_{i, k}\bigg|\\
&\le \bigg(\frac{\eta}{1-\eta}+1\bigg)\bigg|\bigcup_{k=1}^\ell \sigma_n(F_k)\cC_{i, k}\bigg|\\
&\le \frac{1}{(1-\eta)^2}\frac{1-\tau'}{24}d_n+\frac{\eta}{1-\eta}d_n,
\end{align*}
and thus
\[
|\cV_i|= |\cB|-|\cB\setminus \cV_i|\ge (1-\theta)d_n.
\]
Take $\delta>0$ with $2\eta\delta+2\eta\delta \ell'+2\delta\ell'\le \eta$.
Taking $\cV=\cV_i$ in Lemma~\ref{L-Rokhlin1}, when $n\in V$ is large enough, we find $\cC_{i, 1}', \dots, \cC_{i, \ell'}'\subseteq \cV_i$ such that
\begin{enumerate}
\item for every $k=1, \dots, \ell'$ and $c\in \cC_{i, k}'$, the map $s\mapsto \sigma_{n, s}(c)$ from $F'_k$ to $\sigma_n(F'_k)c$ is bijective,

\item the sets $\sigma_n(F'_1)\cC_{i, 1}', \dots, \sigma_n(F'_{\ell'})\cC_{i, \ell'}'$ are pairwise disjoint
and the family $\bigcup_{k=1}^{\ell'}\{\sigma_n(F'_k)c: c\in \cC_{i, k}'\}$ is $\eta$-disjoint and $(1-\theta-\eta)$-covers $\{1, \dots, d\}$,

\item $\sum_{k=1}^{\ell'}||\sigma_n(F'_k)\cC_{i, k}'|/d_n-\lambda_k|<\delta$.
\end{enumerate}

Note that
\[
\sum_{k=1}^{\ell'}\lambda_k\le \frac{1}{d_n} \bigg|\bigcup_{k=1}\sigma_n(F'_k)\cC_{1, k}' \bigg|+\delta\le 1+\delta.
\]
For each $1\le k\le \ell'$, take $\cC''_{1, k}\subseteq \cC'_{1, k}$ and $\cC''_{2, k}\subseteq \cC'_{2, k}$ with
\[
|\cC''_{1, k}|=|\cC''_{2, k}|=\min(|\cC'_{1, k}|, |\cC'_{2, k}|).
\]
Take a bijection $\varphi'_k: \cC''_{1, k}\rightarrow \cC''_{2, k}$.
We have
\begin{align*}
|\cC'_{1, k}|\cdot |F'_k|-|\cC'_{2, k}|\cdot |F'_k| &\le \frac{1}{1-\eta}|\sigma_n(F'_k)\cC'_{1, k}|-|\sigma_n(F'_k)\cC'_{2, k}|\\
&\le (1+2\eta)|\sigma_n(F'_k)\cC'_{1, k}|-|\sigma_n(F'_k)\cC'_{2, k}|\\
&\le (1+2\eta)(\lambda_k+\delta)d_n-(\lambda_k-\delta)d_n\\
&= (2\eta\lambda_k+2\eta\delta+2\delta)d_n,
\end{align*}
and similarly $|\cC'_{2, k}|\cdot |F'_k|-|\cC'_{1, k}|\cdot |F'_k|\le (2\eta\lambda_k+2\eta\delta+2\delta)d_n$.
Thus  for each $i=1, 2$ one has
\begin{align*}
|\sigma_n(F'_k)(\cC'_{i, k}\setminus \cC''_{i, k})|
&\le |\cC'_{i, k}\setminus \cC''_{i, k}|\cdot |F'_k| \\
&=||\cC'_{1, k}|\cdot |F'_k|-|\cC'_{2, k}|\cdot |F'_k||\le (2\eta\lambda_k+2\eta\delta+2\delta)d_n,
\end{align*}
and hence
\begin{align*}
\bigg|\bigcup_{k=1}^{\ell'}\sigma_n(F'_k)(\cC'_{i, k}\setminus \cC''_{i, k})\bigg|&=\sum_{k=1}^{\ell'}|\sigma_n(F'_k)(\cC'_{i, k}\setminus \cC''_{i, k})|\\
&\le \sum_{k=1}^{\ell'}(2\eta\lambda_k+2\eta\delta+2\delta)d_n\\
&\le (2\eta(1+\delta)+2\eta\delta \ell'+2\delta\ell')d_n\le 3\eta d_n.
\end{align*}
Therefore
\begin{align*}
\bigg|\bigcup_{k=1}^{\ell'}\sigma_n(F'_k)\cC''_{i, k}\bigg|
&\ge \bigg|\bigcup_{k=1}^{\ell'}\sigma_n(F'_k)\cC'_{i, k}\bigg|-\bigg|\bigcup_{k=1}^{\ell'}\sigma_n(F'_k)(\cC'_{i, k}\setminus \cC''_{i, k})\bigg| \\
&\ge (1-\theta-\eta)d_n-3\eta d_n=(1-\theta-4\eta)d_n
\end{align*}
for $i=1, 2$.

Since the families $\bigcup_{k=1}^\ell\{\sigma_n(F_k)c: c\in \cC_{i, k}\}$ for $i=1, 2$ are $\eta$-disjoint,
we can find $F_{i, c}\subseteq F_k$ with $|F_{i, c}|\ge (1-\eta)|F_k|$ for all $i=1, 2$, $k=1,\dots, \ell$, and $c\in \cC_{i, k}$
so that for each $i=1, 2$, the sets $\sigma_n(F_{i, c})c$ for $c\in \bigcup_{k=1}^\ell \cC_{i, k}$ are pairwise disjoint.
For every $k=1,\dots,\ell$ and $c\in \cC_{1, k}$, set $\bar{F}_c=F_{1, c}\cap F_{2, \varphi_k(c)}$
and $\hat{F}_c=\{s\in \bar{F}_c: Fs\subseteq \bar{F}_c\}$. Then $|\bar{F}_c|\ge (1-2\eta)|F_k|$
and $|\hat{F}_c|\ge |\tilde{F}_k|-2\eta |F|\cdot |F_k|\ge (1-(2|F|+1)\eta)|F_k|$.

Similarly, for every $1\le k\le \ell'$ and $c\in \cC''_{1, k}$, we find some $\tilde{F}'_c\subseteq F'_k$ with $|\tilde{F}'_c|\ge (1-2\eta)|F'_k|$
such that the sets $\sigma_n(\tilde{F}'_c)c$  for $c\in \bigcup_{k=1}^\ell \cC''_{1, k}$, as well as the sets
$\sigma_n(\tilde{F}'_c)\varphi'_k(c)$ for $c\in \bigcup_{k=1}^\ell \cC''_{1, k}$, are pairwise disjoint.
Setting $\hat{F}'_c=\{s\in \bar{F}'_c: Fs\subseteq \bar{F}'_c\}$, we have $|\hat{F}'_c|\ge (1-(2|F|+1)\eta)|F'_k|$.

Note that
\begin{align*}
\bigg|\bigcup_{k=1}^\ell\bigcup_{c\in \cC_{1, k}}\sigma_n(\hat{F}_c)c\bigg|&=\sum_{k=1}^\ell\sum_{c\in \cC_{1, k}}|\hat{F}_c|\\
&\ge \sum_{k=1}^\ell\sum_{c\in \cC_{1, k}}\big(1-(2|F|+1)\eta\big)|F_k|\\
&\ge \big(1-(2|F|+1)\eta\big)\bigg|\bigcup_{k=1}^\ell\sigma_n(F_k)\cC_{1, k}\bigg|\\
&\ge \big(1-(2|F|+1)\eta\big)(1-\eta)\frac{1-\tau'}{24}d_n.
\end{align*}
Similarly,
\begin{align*}
\bigg|\bigcup_{k=1}^{\ell'}\bigcup_{c\in \cC''_{1, k}}\sigma_n(\hat{F}'_c)c \bigg|
&\ge \big(1-(2|F|+1)\eta\big)\bigg|\bigcup_{k=1}^{\ell'}\sigma_n(F'_k)\cC''_{1, k}\bigg|\\
&\ge \big(1-(2|F|+1)\eta\big)(1-\theta-4\eta)d_n.
\end{align*}
Set $\cW=\big(\bigcup_{k=1}^\ell\bigcup_{c\in \cC_{1, k}}\sigma_n(\hat{F}_c)c\big)\cup\big(\bigcup_{k=1}^{\ell'}\bigcup_{c\in \cC''_{1, k}}\sigma_n(\hat{F}'_c)c\big)$.
Then
\[
|\cW|\ge \big(1-(2|F|+1)\eta\big)(1-\eta)\frac{1-\tau'}{24}d_n+\big(1-(2|F|+1)\eta\big)(1-\theta-4\eta)d_n.
\]

Take a $\varphi\in \Sym(d_n)$ such that $\varphi(\sigma_{n, s}(c))=\sigma_{n, s}(\varphi_k(c))$ for all $k=1,\dots,\ell$, $c\in \cC_{1, k}$, and $s\in \bar{F}_c$, and
$\varphi(\sigma_{n, s}(c))=\sigma_{n, s}(\varphi'_k(c))$ for all $k=1,\dots,\ell'$, $c\in \cC''_{1, k}$, and $s\in \bar{F}'_c$. For every $s\in F$, note that
$\sigma_{n, s}\varphi=\varphi\sigma_{n, s}$ on $\cW$, and hence
\begin{align*}
\rho_{\rm Hamm}(\sigma_{n, s}\varphi, \varphi\sigma_{n, s})
&\le 1-\frac{|\cW|}{d_n} \\
&\le 1-\big(1-(2|F|+1)\eta\big)(1-\eta)\frac{1-\tau'}{24}-\big(1-(2|F|+1)\eta\big)(1-\theta-4\eta)<\eps
\end{align*}
when $\eta$ is small enough. We also have
\begin{align*}
|\varphi(\cY_n)\cap \cZ_n|&\ge \sum_{k=1}^\ell\sum_{c\in \cC_{1, k}}|\{s\in \bar{F}_c:\sigma_{n, s}(c)\in \cY_n, \sigma_{n, s}(\varphi_k(c))\in \cZ_n\}|\\
&\ge\sum_{k=1}^\ell\sum_{c\in \cC_{1, k}}(|\{s\in F_k:\sigma_{n, s}(c)\in \cY_n, \sigma_{n, s}(\varphi_k(c))\in \cZ_n\}|-2\eta|F_k|)\\
&\ge \sum_{k=1}^\ell\sum_{c\in \cC_{1, k}}\bigg(\bigg(\frac{\tau}{2}\bigg)^2|F_k|-2\eta|F_k|\bigg)\\
&\ge \sum_{k=1}^\ell\sum_{c\in \cC_{1, k}}\frac{\tau^2}{8}|F_k|
\ge \frac{\tau^2}{8}\bigg|\bigcup_{k=1}^\ell\sigma_n(F_k)\cC_{1, k}\bigg|\\
&\ge \frac{\tau^2}{8}(1-\eta)\frac{1-\tau'}{24} d_n
\ge \frac{\tau^2}{16}\frac{1-\tau'}{24}d_n=\lambda d_n
\end{align*}
when $\eta$ is small enough.
\end{proof}

\begin{question}
Does every countable sofic group $G$ admit a sofic approximation sequence $\Sigma$ such that
the action of $G'$ on $(\prod_\fF\{1, \dots, d_i\}, \fB, \mu)$ is ergodic?
\end{question}

\section{IE-tuples and algebraic actions}\label{S-algebraic}

By an {\it algebraic action} we mean an action of a countable discrete group $G$ on a compact
metrizable Abelian group $X$ by (continuous) automorphisms. The structure of such an action is
captured by the Pontryagin dual $\widehat{X}$ viewed as a module over the
integral group ring $\Zb G$. The ring $\Zb G$ consists of the finitely supported $\Zb$-valued
functions on $G$, which we write in the form $\sum_{s\in G} f_s s$,
with addition $(\sum_{s\in G} f_s s )+(\sum_{s\in G} g_s s) = \sum_{s\in G} (f_s + g_s)s$
and multiplication $(\sum_{s\in G} f_s s )(\sum_{s\in G} g_s s) = \sum_{s\in G} (\sum_{t\in G} f_t g_{t^{-1} s} )s$.

Given a matrix $A$ in $M_n(\Zb G)$, the left action of $G$ on
$(\Zb G)^n/(\Zb G)^n A$ gives rise via Pontryagin duality to the algebraic action
$G\curvearrowright X_A := \widehat{(\Zb G)^n/(\Zb G)^nA}$.
Write $A^*$ for the matrix in $M_n(\Zb G)$ whose $(i,j)$ entry is the result of applying the involution
$(\sum_{s\in G} f_s s )^* = \sum_{s\in G} f_s s^{-1}$ to the $(j,i)$ entry of  $A$.
Viewing $\widehat{(\Zb G)^n}$ as $((\Rb /\Zb )^G )^n$, we can then identify $X_A$ with
the closed $G$-invariant subset
\[
\big\{ x\in ((\Rb /\Zb )^G )^n : xA^* = 0_{((\Rb /\Zb )^G)^n} \big\}
\]
of $((\Rb /\Zb )^G )^n$ equipped with the action of $G$ by left translation.
In the case that $A$ is invertible in $M_n(\ell^1 (G))$ the action $G\curvearrowright X_A$ is expansive,
and in fact such actions and their restrictions to closed $G$-invariant subgroups
constitute precisely all of the expansive algebraic actions \cite[Thm.\ 3.1]{ChuLi11}.
When $G$ is amenable, given an action of the form
$G\curvearrowright X_A$ with $A$ invertible in $M_n (\ell^1 (G))$, every
tuple of points in $X$ is an IE-tuple (see Lemma~5.4 and Theorems~7.3 and 7.8 in \cite{ChuLi11}).
We will extend this result in two ways in Theorems~\ref{T-invertible to dense} and
\ref{T-invertible to UPE}, which demonstrate that
in broader contexts independent behaviour similarly saturates the structure
of actions of the form $G\curvearrowright X_A$ with $A$ invertible in $M_n(\ell^1 (G))$.

First however we examine orbit IE-tuples in the context of actions $G\curvearrowright X$
on a compact metrizable (not necessarily  Abelian) group by automorphisms.
It was shown in \cite[Theorem 7.3]{ChuLi11} that, when $G$ is amenable, the IE-tuples for such an action
are determined by a closed $G$-invariant normal subgroup of $X$ called the {\it IE group}.
We now proceed to record some observations showing that the basic theory of the IE-group
from \cite{ChuLi11} can be extended from amenable $G$ to general $G$ using orbit IE-tuples.
Thus $G$ will be an arbitrary countable discrete group until we turn to the sofic
setting in Theorem~\ref{T-invertible to UPE}.

The proof of Lemma~3.11 in \cite{KerLi07} shows the following:

\begin{lemma}\label{L-measure to density}
Let $G$ act continuously on a compact metrizable space $X$.
Let $A$ be a Borel subset of $X$, and $\mu$ a $G$-invariant Borel probability measure on $X$.
Then $A$ has independence density at least $\mu(A)$ over $G$.
\end{lemma}

From Lemma~\ref{L-measure to density} we immediately obtain:

\begin{lemma} \label{L-support to orbit}
Let $G$ act continuously on a compact metrizable space $X$. Let $\mu$ be a $G$-invariant Borel probability measure on $X$.
Then every point in the support of $\mu$ is an orbit $\IE$-$1$-tuple.
\end{lemma}

We now suppose that $G$ acts continuously on a compact metrizable group $X$ by automorphisms.
From Lemma~\ref{L-support to orbit} we have:

\begin{lemma} \label{L-orbit 1 tuple}
Every point of $X$ is an orbit $\IE$-$1$-tuple.
\end{lemma}

Denote by $\IE(X)$ the set of all $x\in X$ such that $(x, e_X)$ is an orbit $\IE$-pair, where $e_X$ is the identity
element of $X$. The proof of Theorem~7.3 in \cite{ChuLi11} shows the following.

\begin{theorem} \label{T-IE group}
$\IE(X)$ is a closed $G$-invariant normal subgroup of $X$. For every $k\in \Nb$ the set $\IE_k(X, G)$
of all orbit IE-$k$-tuples is a closed $G$-invariant subgroup of the group $X^k$ and
\begin{align*}
\IE_k(X, G)&=\{(x_1y, \dots, x_ky): x_1, \dots, x_k\in \IE(X),\ y\in X\}\\
&=\{(yx_1, \dots, yx_k): x_1, \dots, x_k\in \IE(X),\ y\in X\}.
\end{align*}
\end{theorem}

Now we suppose that $X$ is Abelian. In this case a point $x\in X$ is said to be
{\it $1$-homoclinic} if the function $s\mapsto \varphi (sx) - 1$ on $G$ lies
in $\ell^1 (G)$ for every $\varphi$ in the Pontryagin dual $\widehat{X}$.
The set of $1$-homoclinic points is written $\Delta^1(X)$. This set was studied in
\cite{LSV, LSV12,SV} in the case $G=\Zb^d$ and in \cite{ChuLi11} for more
general $G$.
From the proof of Theorem~7.8 in \cite{ChuLi11} we obtain the following.

\begin{theorem}\label{T-homoclinic to IE}
Suppose that  $\widehat{X}$ is a finitely generated left $\Zb G$-module.
Then $\Delta^1(X)\subseteq \IE(X)$.
\end{theorem}

From Theorem~\ref{T-homoclinic to IE}  and \cite[Lemma 5.4]{ChuLi11} we obtain:

\begin{theorem}\label{T-invertible to dense}
Let $n\in\Nb$, and let $A$ be an element of $M_n(\Zb G)$ which is invertible in $M_n(\ell^1(G))$.
Then for the action $G\curvearrowright X_A$ one has $\IE(X_A)=X_A$.
\end{theorem}

Now we let $G$ be a countable sofic group and
$\Sigma=\{\sigma_i: G\rightarrow \Sym(d_i)\}_{i=1}^\infty$ a sofic approximation sequence for $G$.

\begin{theorem}\label{T-invertible to UPE}
Let $n \in \Nb$, and let $A$ be an element of $M_n(\Zb G)$ which is invertible in $M_n(\ell^1(G))$.
Consider the action $G\curvearrowright X_A$. Then,
for each $k\in \Nb$, every $k$-tuple of points in $X_A$ is a $\Sigma$-$\IE$-tuple.
\end{theorem}

Before proceeding to the proof of Theorem~\ref{T-invertible to UPE}, we give an application
to a problem of Deninger. For an invertible element $f$ in the group von Neumann algebra $\cL G$
of a countable discrete group $G$ the Fuglede-Kadison determinant is defined
by $\det_{\cL G} f = \exp \tr (\log |f|)$ where $\tr$ is the canonical trace on $\cL G$.

In \cite[Question 26]{Den09} Deninger asked whether $\det_{\cL G} f>1$ whenever
$f\in \Zb G$ is invertible in $\ell^1(G)$ and has no left inverse in $\Zb G$.
An affirmative answer was given by Deninger and Schmidt
in the case that $G$ is residually finite and amenable \cite[Cor.\ 6.7]{DenSch07} and more
generally by Chung and the second author in the case $G$ is
amenable \cite[Corollary 7.9]{ChuLi11}.
Using Theorem~\ref{T-invertible to UPE}, Proposition~\ref{P-basic}(3),
Theorem~7.1 in \cite{KerLi11}, and the argument in the proof of Corollary~7.9 in \cite{ChuLi11},
we obtain an answer to Deninger's question for all countable residually finite groups:

\begin{corollary} \label{C-answer to Deninger}
Suppose that $G$ is residually finite and that $f$ is an element of $\Zb G$ which is invertible in
$\ell^1(G)$ and has no left inverse in $\Zb G$. Then $\det_{\cL G} f>1$.
\end{corollary}

Let $n\in \Nb$.  For $A=(A_{ij})_{1\le i, j\le n}\in M_n(\ell^1(G))$, we set
\[
\|A\|_1=\sum_{1\le i, j\le n}\|A_{ij}\|_1.
\]
For $(a_1,\dots, a_n)\in \Rb^d$, we set $\|(a_1, \dots, a_n)\|_\infty=\max_{1\le j\le n}|a_j|$. For $\xi: \{1, \dots, d\}\rightarrow \Zb^n$, we set
\[
\|\xi\|_\infty=\max_{1\le j\le d}\|\xi(j)\|_\infty.
\]
Denote by $P$ the natural quotient map $(\Rb^n)^G\rightarrow ((\Rb/\Zb)^n)^G$.
Denote by $\rho$ the canonical metric on $\Rb/\Zb$ defined by
\[
\rho(t_1+\Zb, t_2+\Zb):=\min_{m\in \Zb}|t_1-t_2-m|.
\]
By abuse of notation, we also use $\rho$ to denote the metric on $(\Rb/\Zb)^n$ defined by
\[
\rho((a_1, \dots, a_n), (b_1, \dots, b_n)):=\max_{1\le j\le n}\rho(a_j, b_j).
\]
Via the coordinate map at the identity element of $G$,
we will think of $\rho$ as a continuous pseudometric on $((\Rb/\Zb)^n)^G$.

\begin{lemma} \label{L-point}
Let $n \in \Nb$, and let $A$ be an element of $M_n(\Zb G)$ which is invertible in $M_n(\ell^1(G))$.
Consider the action $G\curvearrowright X_A$.
Let $F$ be a nonempty finite subset of $G$ and let $M, \delta>0$. For every $d\in \Nb$,
good enough sofic approximation $\sigma: G\rightarrow \Sym(d)$, and
$\xi: \{1, \dots, d\}\rightarrow \Zb^n$ with $\|\xi\|_\infty\le M$,
if we define $h:\{1, \dots, d\}\rightarrow (\Zb^n)^G$ and $\varphi: \{1, \dots, d\}\rightarrow X_A$ by
\[
(h(a))_{t^{-1}}=\xi(ta) \text{ for all } t\in G
\]
and
\[
\varphi(a)=P((h(a))(A^*)^{-1})
\]
then $\varphi\in \Map(\rho, F, \delta, \sigma)$.
\end{lemma}

\begin{proof}
Since $(A^*)^{-1}\in M_d(\ell^1(G))$, there exists a nonempty finite subset
$K$ of $G$ such that for all $z_1, z_2\in (\Zb^n)^G$ such that
$\|z_1\|_\infty, \|z_2\|_\infty\le M$ and $z_1, z_2$ coincide on $K$, one has
$\|(z_1(A^*)^{-1})_e-(z_2(A^*)^{-1})_e\|_\infty<\delta/2$, which implies
that $\rho(P(z_1(A^*)^{-1}), P(z_2(A^*)^{-1}))<\delta/2$.

Denote by $\Lambda$ the set of all $a\in \{1, \dots, d\}$ satisfying $t(sa)=(ts)a$ for all $t\in K^{-1}$ and $s\in F$.
When $\sigma$ is a good enough sofic approximation for $G$, one has $|\Lambda|\ge (1-(\delta/2)^2)d$.
Let $a\in \Lambda$ and $s\in F$. Then
\[
s(\varphi(a))=P((s(h(a)))(A^*)^{-1})
\]
and
\[
\varphi(sa)=P((h(sa))(A^*)^{-1}).\]
For every $t\in K^{-1}$ one has
\[
(s(h(a)))_{t^{-1}}=(h(a))_{s^{-1}t^{-1}}=\xi((ts)a)=\xi(t(sa))=(h(sa))_{t^{-1}}.
\]
Thus, by the choice of $K$, we have $\rho(s(\varphi(a)), \varphi(sa))<\delta/2$.
Note that $((\Rb/\Zb)^n)^G$ has diameter $1$ under $\rho$. It follows that
\[
\rho_2(s\varphi(\cdot), \varphi(s\cdot))\le ((\delta/2)^2+1-|\Lambda|/d)^{1/2}<\delta.
\]
Therefore $\varphi\in \Map(\rho, F, \delta, \sigma)$.
\end{proof}

We are ready to prove Theorem~\ref{T-invertible to UPE}.

\begin{proof}[Proof of Theorem~\ref{T-invertible to UPE}]
Let $\ox=(x_1,\dots, x_k)$ be a $k$-tuple of points in $X_A$.
Then for each $1\le j\le k$ there is a $z_j\in (\Zb^n)^G$ such that
$\|z_j\|_\infty\le \|A\|_1$ and $x_j=P(z_j(A^*)^{-1})$.

Let $U_1\times \cdots \times U_k$ be a product neighborhood of $\ox$ in $X^k$.
Since the map from bounded subsets of $(\Zb^n)^G$ equipped with the pointwise convergence topology
to $X_A$ sending $z$ to $P(z(A^*)^{-1})$ is continuous \cite[Prop.\ 4.2]{DenSch07},
there is a nonempty finite subset $K$ of $G$ such that for every $1\le j\le k$ and
$z\in (\Zb^n)^G$ with $\|z\|_\infty\le \|A\|_1$ and $z|_K=z_j|_K$, one has $P(z(A^*)^{-1})\in U_j$.

Let $F$ be a nonempty finite subset of $G$ and $\delta>0$. Let $d\in \Nb$ and
let $\sigma$ be a map from $G$ to $\Sym(d)$.
Denote by $\Lambda$ the set of all $a\in \{1, \dots, d\}$ such that $sa\neq ta$ for all distinct
$s, t\in K^{-1}$. When $\sigma$ is a good enough sofic approximation for $G$,
we have $|\Lambda|\ge d/2$. Let $\cJ$ be a maximal subset of $\Lambda$ subject to the
condition that the sets $K^{-1}a$ for $a\in \cJ$ are pairwise disjoint. Then $\Lambda\subseteq (\sigma(K^{-1}))^{-1}\sigma(K^{-1})\cJ$, and hence $|\Lambda|\le |K|^2|\cJ|$.
Therefore $|\cJ|\ge d/(2|K|^2)$.

We claim that $\cJ$ is a $(\rho, F, \delta, \sigma)$-independence set for $\oU=(U_1, \dots, U_k)$
when $\sigma$ is a good enough sofic approximation for $G$. Let $\omega$ be a map from
$\{1, \dots, d\}$ to $\{1, \dots, k\}$. Define $\xi: \{1, \dots, d\}\rightarrow \Zb^n$
by $\xi(ta)=(z_{\omega(a)})_{t^{-1}}$ for all $a\in \cJ$ and $t\in K^{-1}$, and $\xi(b)=0$
for all $b$ not in $K^{-1}\cJ$. Then we have $h:\{1, \dots, d\}\rightarrow (\Zb^n)^G$ and
$\varphi\in \Map(\rho, F, \delta, \sigma)$ defined in Lemma~\ref{L-point} for $M=\|A\|_1$
when $\sigma$ is a good enough sofic approximation. Let $a\in \cJ$. For any $t\in K^{-1}$,
one has
\[
(h(a))_{t^{-1}}=\xi(ta)=(z_{\omega(a)})_{t^{-1}}.
\]
By the choice of $K$ we have $\varphi(a)=P(h(a)(A^*)^{-1})\in U_{\omega(a)}$.
This proves our claim, and finishes the proof of the theorem.
\end{proof}

\section{Orbit $\IE$-tuples and untameness}\label{S-untame}

Let $G$ be a countably infinite group acting continuously on a compact Hausdorff space $X$.

\begin{theorem}\label{T-orbit IE to nontame}
Let $k\in\Nb$ and let $\oA$ be a $k$-tuple of subsets of $X$.
Suppose that $\oA$ has positive independence density over $G$.
Then $\oA$ has an infinite independence set in $G$.
\end{theorem}

\begin{proof}
Denote by $q$ the density of $\oA$ over $G$.

Let $F_1$ be a nonempty finite subset of $G$. Take $s_1, s_2, \dots$ in $G$ such that setting
$F_{n+1}=F_n \cup F_ns_n$ for all $n\in \Nb$ one has
$F_n\cap F_ns_n=\emptyset$ for all $n\in \Nb$.

Let $n\in \Nb$. Take an independence set $E_n$ of $\oA$ contained in $F_n$ with $|E_n|\ge q|F_n|$. We will construct, inductively on $m$, nonempty finite subsets $F^{(n)}_{m, k}$ and $E^{(n)}_m$ of $G$
for all $1\le m\le k\le n$ and $t^{(n)}_m\in G$ for all $1\le m<n$ such that
\begin{enumerate}
\item $F^{(n)}_{n, n}=F_n$ and $E^{(n)}_n=E_n$;

\item $t^{(n)}_m$ is equal to either $e$ or $s_m^{-1}$ for each $1\le m<n$;

\item $F^{(n)}_{m, k}=F^{(n)}_{m+1, k}t^{(n)}_m$ and $F^{(n)}_{m, m}=F_m$ for all $1\le m< k\le n$;

\item $E^{(n)}_m=E^{(n)}_{m+1}t^{(n)}_m$ for each $1\le m<n$;

\item $|E^{(n)}_m\cap F^{(n)}_{m, k}|\ge q|F^{(n)}_{m, k}|$ for all $1\le m\le k\le n$.
\end{enumerate}
To start with, we define $F^{(n)}_{n, n}$ and $E^{(n)}_n$ according to (1). If $|E^{(n)}_n\cap F_{n-1}|\ge q|F_{n-1}|$, we set $t^{(n)}_{n-1}=e$.
Otherwise, since $|E^{(n)}_n\cap F^{(n)}_{n, n}|\ge q|F^{(n)}_{n, n}|$ and $F^{(n)}_{n, n}=F_n$ is the disjoint union of $F_{n-1}$ and $F_{n-1}s_{n-1}$, we must
have $|E^{(n)}_n\cap F_{n-1}s_{n-1}|\ge q|F_{n-1}s_{n-1}|$, and we set $t^{(n)}_{n-1}=s^{-1}_{n-1}$. Defining $F^{(n)}_{n-1, k}$ for $n-1\le k\le n$ and $E^{(n)}_{n-1}$ according
to (3) and (4) respectively, we have that (5) holds for $m=n-1$. Next,
if $|E^{(n)}_{n-1}\cap F_{n-2}|\ge q|F_{n-2}|$, we set $t^{(n)}_{n-2}=e$.
Otherwise, since $|E^{(n)}_{n-1}\cap F^{(n)}_{n-1, n-1}|\ge q|F^{(n)}_{n-1, n-1}|$ and $F^{(n)}_{n-1, n-1}=F_{n-1}$ is the disjoint union of $F_{n-2}$ and $F_{n-2}s_{n-2}$, we must
have $|E^{(n)}_{n-1}\cap F_{n-2}s_{n-2}|\ge q|F_{n-2}s_{n-2}|$, and
we set $t^{(n)}_{n-2}=s^{-1}_{n-2}$. Defining $F^{(n)}_{n-2, k}$ for $n-2\le k\le n$ and $E^{(n)}_{n-2}$
according to (3) and (4) respectively, we have that (5) holds for $m=n-2$. Continuing in this way, we define
$F^{(n)}_{m, k}, E^{(n)}_m$, and $t^{(n)}_m$ satisfying the above conditions.

Note that if $E'$ is an independence set for $\oA$ in $G$, then $E's$ is an independence set for $\oA$ in $G$ for all $s\in G$.
By induction on $m$, we find easily that $E^{(n)}_m$ is an independence set for $\oA$ in $G$ for all
$n\in \Nb$ and $1\le m\le n$. Also note that for any $1\le m\le k\le n$, $F^{(n)}_{m, k}$ depends only on
$F_k\cap E^{(n)}_k$. In particular, for any fixed $k\in \Nb$ the number of sets appearing in
$F^{(n)}_{1, k}$ for all $n\ge k$ is finite. Thus we can find a strictly increasing sequence $n_1<n_2<\dots $
in $\Nb$ such that for any fixed $k\in \Nb$ the sets $F^{(n_l)}_{1, k}$ and
$E^{(n_l)}_1\cap F^{(n_l)}_{1, k}$ do not depend on $l\ge k$.
Set $E=\bigcup_{k\in \Nb}(E^{(n_k)}_1\cap F^{(n_k)}_{1, k})$.
Since $|E^{(n_k)}_1\cap F^{(n_k)}_{1, k}|\ge q|F^{(n_k)}_{1, k}|=q|F_k|=q|F_1|2^{k-1}$
for every $k\in \Nb$, the set $E$ is infinite. For evey $k\in \Nb$ one has
\begin{align*}
(E^{(n_{k+1})}_1\cap F^{(n_{k+1})}_{1, k+1})\cap F^{(n_k)}_{1, k}&=(E^{(n_{k+1})}_1
\cap F^{(n_{k+1})}_{1, k+1})\cap F^{(n_{k+1})}_{1, k}\\
&=E^{(n_{k+1})}_1\cap F^{(n_{k+1})}_{1, k}\\
&=E^{(n_k)}_1\cap F^{(n_k)}_{1, k}.
\end{align*}
Thus the sequence $\{E^{(n_k)}_1\cap F^{(n_k)}_{1, k}\}_{k\in \Nb}$ is increasing. Since the family of
independence sets for $\oA$ in $G$ is closed under taking increasing unions, we conclude that $E$ is an
independence set for $\oA$ in $G$.
\end{proof}

Recall that that a tuple $(x_1 , \dots , x_k ) \in X^k$ is an {\it IT-tuple} if for every product
neighbourhood $U_1 \times\cdots\times U_k$ of $(x_1 , \dots , x_k )$ the tuple $(U_1 ,\dots, U_k )$
has an infinite independence set \cite{KerLi07}.

\begin{corollary}\label{C-orbit IE to IT}
Every orbit IE-tuple of the action $G\curvearrowright X$ is an IT-tuple.
\end{corollary}

Write $C(X)$ for the Banach space of continuous complex-valued functions on $X$
with the supremum norm.
The action $G\curvearrowright X$ is said to be {\it tame} if no element $f\in C(X)$
admits an infinite subset $J$ of $G$ such that, for $s$ ranging in $J$, the family of functions
$x\mapsto f(s^{-1} x)$ in $C(X)$ is equivalent to the standard basis of $\ell^1$,
meaning that there is a bijection between the two which extends to an isomorphism
(i.e., a bounded linear map with bounded inverse) between
the closures of their linear spans \cite{Koh95,Gla06}.
The action is tame if and only if there is no nondiagonal IT-pair in $X\times X$ \cite[Prop.\ 6.4]{KerLi07}.
Thus from the above corollary we see that a tame action has no nondiagonal orbit IE-tuples.

\section{$\Sigma$-IE-tuples and Li-Yorke Chaos}\label{S-chaos}

Let $G$ be a countably infinite sofic group and
$\Sigma=\{\sigma_i: G\rightarrow \Sym(d_i)\}_{i=1}^\infty$ a sofic approximation sequence for $G$.
We fix a free ultrafilter $\fF$ on $\Nb$ and use it in the definitions of sofic topological entropy
and $\Sigma$-IE-tuples, as in Section~\ref{S-product}.

Let $G\curvearrowright X$ be a continuous action on a compact metrizable space.
Let $\rho$ be a compatible metric on $X$. We say that $(x, y)\in X\times X$ is a {\it Li-Yorke pair} if
\[
\limsup_{G \ni s\to \infty}\rho(sx, sy)>0 \hspace*{3mm}\text{and} \hspace*{3mm} \liminf_{G \ni s\to \infty}\rho(sx, sy)=0.
\]
where the limit supremum and limit infimum
mean the limits of $\sup_{s\in G\setminus F} \rho(sx, sy)$ and\linebreak $\inf_{s\in G\setminus F} \rho(sx, sy)$, respectively,
over the net of finite subsets $F$ of $G$.
Note that the definition of Li-Yorke pair does not depend on the choice of the metric $\rho$.
We say that the action $G\curvearrowright X$ is {\it Li-Yorke chaotic} if there is an uncountable
subset $Z$ of $X$ such that every nondiagonal pair $(x, y)$ in $Z\times Z$ is a Li-Yorke pair.
These definitions adapt those for continuous $\Nb$-actions, which have their origins
in \cite{LiYor75}. In that setting Blanchard, Glasner, Kolyada, and Maass showed that positive entropy
implies Li-Yorke chaos \cite{BlaGlaKolMaa02}. The following theorem demonstrates that, in our sofic context,
positive topological entropy with respect to some sofic approximation sequence implies
Li-Yorke chaos (cf.\ \cite[Thm.\ 3.18]{KerLi07}).

\begin{theorem}\label{T-positive entropy to chaos}
Suppose that
$k\ge 2$ and $\ox=(x_1, \dots, x_k)$ is a $\Sigma$-$\IE$-tuple in
$X^k$ with $x_1, \dots, x_k$ pairwise distinct. For each $1\le j\le k$, let $A_j$ be a
neighbourhood of $x_j$. Then there exist Cantor
sets $Z_j\subseteq A_j$ for $j=1,\dots ,k$ such that the
following hold:
\begin{enumerate}
\item every nonempty finite
tuple of points in $Z:=\bigcup_jZ_j$ is a $\Sigma$-$\IE$-tuple;

\item for all $m\in \Nb$, distinct $y_1, \dots, y_m \in Z$,
and $y'_1, \dots, y'_m \in Z$ one has
\[ \liminf_{G\ni s\to \infty}\max_{1\le i\le m} \rho(sy_i, y'_i)=0.\]
\end{enumerate}
\end{theorem}

We now set out to prove Theorem~\ref{T-positive entropy to chaos}. We begin
with the following lemmas.

\begin{lemma}\label{L-LY}
Let $k\ge 2$ and $\oA=(A_1, \dots, A_k)$ be a tuple of closed subsets of $X$ with positive upper independence
density over $\Sigma$.  For each $j=1, \dots, k$ let $U_j$ be an open set containing $A_j$.
Let $E$ be a finite subset of $G$.
Then there exists an $s\in G\setminus E$ such that the tuple $\oA'$ consisting of $A_i\cap s^{-1}U_j$ for all $i, j=1, \dots, k$
has positive upper independence density over $\Sigma$.
\end{lemma}

\begin{proof}
Take $1<\lambda<\frac{k}{k-1}$. Then we have the constant $c>0$ in Lemma~\ref{L-KM}.
Take a $q>0$ such that for every nonempty finite subset $F$ of $G$ and $\delta>0$ the set $V_{F, \delta}$
of all $i\in \Nb$ for which $\oA$ has a $(\rho, F, \delta, \sigma_i)$-independence set of cardinality
at least $qd_i$ is in $\fF$.
Take a finite subset $W$ of $G$ such that $cq|W|>8$ and for any distinct $s, t\in W$
one has $s^{-1}t\not \in E$.
When $0<|W|^2\kappa<1/2$, the number of subsets of $\{1, \dots, d\}$ of cardinality no greater
than $|W|^2\kappa d$
is equal to $\sum^{\lfloor |W|^2\kappa d\rfloor}_{j=0}\binom{d}{j}$, which is at most
$|W|^2\kappa d \binom{d}{|W|^2\kappa d}$,
which by Stirling's approximation is less than $\exp(\beta d)$
for some $\beta > 0$ depending on $\kappa$
but not on $d$ when $d$ is sufficiently large with $\beta\to 0$ as $\kappa\to 0$.
Take $cq/(2|W|^2)>\kappa>0$ such that for any $1\le j\le k$ and $x\in X\setminus U_j$
one has $\rho(x, A_j)>\sqrt{\kappa}$ and for all sufficiently large $d\in \Nb$ the number
of subsets of $\{1, \dots, d\}$
of cardinality no greater than $|W|^2\kappa d$ is at most $\big(\frac{k}{(k-1)\lambda}\big)^{q d}$.

Let $F$ be a nonempty finite subset of $G$ and $\delta>0$. Set $F'=F\cup W$ and $\delta'=\min(\delta, \kappa)$.
Let $i\in\Nb$ be such that $\oA$ has a $(\rho, F', \delta', \sigma_i)$-independence set $\cJ_i$ of cardinality at least $qd_i$.
For each $\omega\in \{1, \dots, k\}^{\cJ_i}$ take a $\varphi_\omega\in \Map(\rho, F', \delta', \sigma_i)$ such that
$\varphi_\omega(a)\in A_{\omega(a)}$ for every $a\in \cJ_i$. For each $\omega\in \{1, \dots, k\}^{\cJ_i}$, there is some
$\Lambda_\omega\subseteq \{1, \dots, d_i\}$ with $|\Lambda_\omega|\ge (1-|W|^2\delta')d_i$ such that
$\rho(\varphi_\omega(\sigma_i(s)a),s\varphi_\omega(a))<\sqrt{\delta'}$ for all
$s\in W^{-1}W$ and $a\in \Lambda_\omega$.
By the choice of $\kappa$, when $i$ is large enough there is a subset $\Omega_i$ of $\{1, \dots, k\}^{\cJ_i}$ with
$\big(\frac{k}{(k-1)\lambda}\big)^{q d_i}|\Omega_i|\ge k^{|\cJ_i|}$
such that the set $\Lambda_\omega$ is the same, say $\Theta_i$, for every $\omega \in \Omega_i$,
and $|\Theta_i|/d_i\ge 1-|W|^2\delta'>1-cq/2$. Then
\[
|\Omega_i|\ge k^{|\cJ_i|}\bigg(\frac{(k-1)\lambda}{k}\bigg)^{q d_i}
\ge k^{|\cJ_i|}\bigg(\frac{(k-1)\lambda}{k}\bigg)^{|\cJ_i|}=((k-1)\lambda)^{|\cJ_i|}.
\]
By our choice of $c$, we can find a subset $\cJ'_i$ of $\cJ_i$ with $|\cJ'_i|\ge c|\cJ_i|\ge cq d_i$  such that every map
$\cJ'_i\rightarrow \{1, \dots, k\}$ extends to some $\omega\in \Omega_i$. When $i$ is large enough, one also has
$|\cW_i|\ge (1-cq/4)d_i$ for the set
\begin{align*}
\cW_i &= \big\{a\in \{1, \dots, d_i\} : ((\sigma_i(s))^{-1}\sigma_i(t))(a)=\sigma_i(s^{-1}t)(a)
\text{ for all } s, t\in W \\&\hspace*{40mm} \text{and } \sigma_i(s)(a)\neq a \text{ for all } s\in W^{-1}W\setminus \{e\} \big\}.
\end{align*}
Note that $|\cW_i\cap \Theta_i\cap \cJ'_i|\ge cqd_i/4$ and every map $\cW_i\cap \Theta_i\cap \cJ'_i\rightarrow \{1, \dots, k\}$
extends to some $\omega\in \Omega_i$.

Denote by $\eta$ the maximum of $|\sigma_i(s)(\cW_i\cap \Theta_i\cap \cJ'_i)\cap \sigma_i(t)(\cW_i\cap \Theta_i\cap \cJ'_i)|/d_i$
for $s, t$ ranging over distinct elements of $W$. Then for each $s\in W$ there is a subset $\Upsilon_{i, s}$ of
$\sigma_i(s)(\cW_i\cap \Theta_i\cap \cJ'_i)$ with cardinality at most $\eta |W|d_i$ such that the sets
$(\sigma_i(s)(\cW_i\cap \Theta_i\cap \cJ'_i))\setminus \Upsilon_{i, s}$ for $s\in W$ are pairwise disjoint.
It follows
that
\begin{align*}
\sum_{s\in W}|\sigma_i(s)(\cW_i\cap \Theta_i\cap \cJ'_i)|
&\le \eta |W|^2d_i
+\bigg|\bigcup_{s\in W}((\sigma_i(s)(\cW_i\cap \Theta_i\cap \cJ'_i))\setminus \Upsilon_{i, s})\bigg| \\
&\le \eta |W|^2d_i+d_i.
\end{align*}
On the other hand, we have
\[
\sum_{s\in W}|\sigma_i(s)(\cW_i\cap \Theta_i\cap \cJ'_i)|=|W|\cdot |\cW_i\cap \Theta_i\cap \cJ'_i|\ge |W|cq d_i/4\ge 2d_i.
\]
Thus $\eta\ge 1/|W|^2$. Then we can  find some distinct $t_i, t'_i\in W$ with
$|\sigma_i(t_i)(\cW_i\cap \Theta_i\cap \cJ'_i)\cap \sigma_i(t'_i)(\cW_i\cap \Theta_i\cap \cJ'_i)|\ge d_i/|W|^2$.
Set $s_i=t_i^{-1}t'_i$. Then $s_i\in W^{-1}W\setminus \{e\}$, and
\begin{align*}
\lefteqn{|(\cW_i\cap \Theta_i\cap \cJ'_i)\cap (\sigma_i(s_i))^{-1}(\cW_i\cap \Theta_i\cap \cJ'_i)|}\hspace*{35mm} \\
\hspace*{30mm} &=|(\cW_i\cap \Theta_i\cap \cJ'_i)\cap \sigma_i(s_i)(\cW_i\cap \Theta_i\cap \cJ'_i)|\\
&=|\sigma_i(t_i)(\cW_i\cap \Theta_i\cap \cJ'_i)\cap \sigma_i(t'_i)(\cW_i\cap \Theta_i\cap \cJ'_i)|\\
&\ge d_i/|W|^2.
\end{align*}
Take a maximal subset $\Xi_i$ of $(\cW_i\cap \Theta_i\cap \cJ'_i)\cap (\sigma_i(s_i))^{-1}(\cW_i\cap \Theta_i\cap \cJ'_i)$
subject to the condition that for any $a\in \Xi_i$,
neither  $\sigma_i(s_i)(a)$ nor $(\sigma_i(s_i))^{-1}(a)$  is in $\Xi_i$. Then
$\Xi_i\cup \sigma_i(s_i)\Xi_i\cup (\sigma_i(s_i))^{-1}\Xi_i\supseteq (\cW_i\cap \Theta_i\cap \cJ'_i)\cap (\sigma_i(s_i))^{-1}(\cW_i\cap \Theta_i\cap \cJ'_i)$.
It follows that $|\Xi_i|\ge |(\cW_i\cap \Theta_i\cap \cJ'_i)\cap (\sigma_i(s_i))^{-1}(\cW_i\cap \Theta_i\cap \cJ'_i)|/3\ge d_i/(3|W|^2)$.
Note that $\Xi_i$ and $\sigma_i(s_i)\Xi_i$ are disjoint subsets of $\cW_i\cap \Theta_i\cap \cJ'_i$.

Let $\xi=(\xi_1, \xi_2): \Xi_i\rightarrow \{1, \dots, k\}^2$.  Define a map
$\xi': \Xi_i\cup \sigma_i(s_i)\Xi_i\rightarrow \{1, \dots, k\}$ by
$\xi'(a)=\xi_1(a)$ and $\xi'(\sigma_i(s_i)(a))=\xi_2(a)$ for all $a\in \Xi_i$. Extend $\xi'$
to some $\omega\in \Omega_i$. Then
$\varphi_\omega \in \Map(\rho, F, \delta, \sigma_i)$, $\varphi_\omega(a)\in A_{\omega(a)}=A_{\xi_1(a)}$ and
$\varphi_\omega(\sigma_i(s_i)(a))\in A_{\omega(\sigma_i(s_i)(a))}=A_{\xi_2(a)}$ for all $a\in \Xi_i$.
For any $a\in \Xi_i$, since $\rho(\varphi_\omega(\sigma_i(s_i)a),s_i\varphi_\omega(a))<\sqrt{\delta'}\le \sqrt{\kappa}$,
by the choice of $\kappa$ we have $s_i\varphi_\omega(a)\in U_{\xi_2(a)}$, and hence
$\varphi_\omega(a)\in A_{\xi_1(a)}\cap s_i^{-1}U_{\xi_2(a)}$. Therefore $\Xi_i$
is a $(\rho, F, \delta, \sigma_i)$-independence set of cardinality at least $d_i/(3|W|^2)$ for the tuple
consisting of $A_l\cap s_i^{-1}U_j$ for all $l, j=1, \dots, k$.

There is some $s_{F, \delta}\in W^{-1}W\setminus \{e\}$ such that the set of $i\in V_{F', \delta'}$
for which $s_i$ is defined and $s_i=s_{F, \delta}$ lies in $\fF$.
It follows that we can find an $s\in W^{-1}W\setminus \{e\}$ such that for any nonempty finite
subset $F$ of $G$ and $\delta>0$ there are some nonempty finite subset $\tilde{F}$ of $G$
and $\tilde{\delta}>0$ with $F\subseteq \tilde{F}$ and $\delta>\tilde{\delta}$ such that
$s_{\tilde{F}, \tilde{\delta}}=s$. Then the tuple $\oA'$ consisting of $A_l\cap s^{-1}U_j$ for all $l, j=1, \dots, k$
has upper independence density at least $1/(3|W|^2)$ over $\Sigma$.
From the choice of $W$ we have $s\not \in E$.
\end{proof}

From Lemma~\ref{L-LY} by induction on $m$ we have:

\begin{lemma}\label{L-LY multiple}
Let $k\ge 2$ and $\oA=(A_1, \dots, A_k)$ be a tuple of closed subsets of $X$ with positive upper independence
density over $\Sigma$.  For each $j=1, \dots, k$ let $U_j$ be an open set containing $A_j$.
Let $E$ be a finite subset of $G$ and $m\in \Nb$. Then there exist $s_1, \dots, s_m\in G\setminus E$ such that
$s_i^{-1}s_j\not \in E$ for all distinct $1\le i, j\le m$ and the tuple $\oA'$ consisting of
$A_i\cap s^{-1}_1U_{\omega(1)}\cap \dots \cap s^{-1}_mU_{\omega(m)}$ for all $1\le i\le k$ and $\omega \in \{1, \dots, k\}^m$
has positive upper independence density over $\Sigma$.
\end{lemma}

We are ready to prove Theorem~\ref{T-positive entropy to chaos}.

\begin{proof}[Proof of Theorem~\ref{T-positive entropy to chaos}] We may assume that the $A_j$
are closed and pairwise disjoint. Take an increasing sequence $E_1\subseteq E_2\subseteq \dots$ of
finite subsets of $G$ with union $G$. We shall construct, via induction on $m$, closed nonempty subsets
$A_{m, j}$ of $X$ for $1\le j\le k^{2^{m-1}}$ with the following properties:
\begin{enumerate}
\item[(a)] $A_{1, j} = A_j$ for all $1\le j\le k$,

\item[(b)] for every $m\ge 2$ and $1\le i\le k^{2^{m-2}}$, $A_{m-1, i}$ contains exactly $k^{2^{m-2}}$
of the $A_{m, j}$ for $1\le j\le k^{2^{m-1}}$,

\item[(c)] for every $m\ge 2$ and map $\gamma:\{1, \dots, k^{2^{m-1}}\}\rightarrow \{1, \dots, k^{2^{m-2}}\}$
there exists a $t_{\gamma}\in G\setminus E_{m-1}$  such that
$t_{\gamma}A_{m, j}\subseteq \overline{U_{m-1, \gamma(j)}}$ for all
$1\le j\le k^{2^{m-1}}$, where $U_{m-1, i}=\{x\in X: \rho(x, A_{m-1, i})<2^{-m}\delta_{m-1}\}$
for all $1\le i\le k^{2^{m-2}}$ and $\delta_{m-1}=\min \rho(x, y)$ for $x, y$ ranging over points in distinct $A_{m-1, j}$,

\item[(d)] when $m\ge 2$, $\diam(A_{m, j})\le 2^{-m}$ for all $1\le j\le k^{2^{m-1}}$,

\item[(e)] for every $m$, the sets $A_{m, j}$ for $1\le j\le k^{2^{m-1}}$ are pairwise disjoint,

\item[(f)] for every $m$, the collection $\{A_{m, j} : 1\le j\le k^{2^{m-1}}$\},
ordered into a tuple, has positive upper independence density over $\Sigma$.
\end{enumerate}
Suppose that we have constructed such $A_{m, j}$ over all $m$.
Properties (b), (d) and (e) imply that $Z=\bigcap_{m\in \Nb}\bigcup_{j=1}^{k^{2^{m-1}}}A_{m, j}$ is a Cantor set.
Property (a) implies that $Z_j:=Z\cap A_j$ is also a Cantor set for each $1\le j\le k$.
Condition (1) follows from properties (d) and (f). Condition
(2) follows from properties (c) and (d).

We now construct the $A_{m, j}$. Define $A_{1, j}$ for $1\le j\le k$ according to property (a).
By assumption properties (e)  and (f) are
satisfied for $m=1$. Assume that we have constructed $A_{m, j}$ for
all $j=1,\dots , k^{2^{m-1}}$ with the above
properties. Set $n=1+(k^{2^{m-1}})^{k^{2^m}}$. By Lemma~\ref{L-LY multiple} we can find
$s_1, \dots, s_n\in G\setminus E_m$ such that  the tuple  consisting of
$A_{m,i}\cap s^{-1}_1U_{m, \omega(1)}\cap \dots \cap s^{-1}_nU_{m, \omega(n)}$ for all $1\le i\le k^{2^{m-1}}$
and $\omega \in \{1, \dots, k^{2^{m-1}}\}^n$ has positive upper independence density over $\Sigma$.
Take a bijection $\varphi: \{1, \dots, k^{2^{m-1}}\}^{\{1, \dots, k^{2^m}\}}\rightarrow \{2, \dots, n\}$.
For each $\gamma:\{1, \dots, k^{2^{m}}\}\rightarrow \{1, \dots, k^{2^{m-1}}\}$, set
$t_\gamma=s_{\varphi(\gamma)}$. For all $1\le i, j\le k^{2^{m-1}}$, define
$\omega_{i, j}\in \{1, \dots, k^{2^{m-1}}\}^n$ by $\omega_{i, j}(1)=j$ and
$\omega_{i, j}(\varphi(\gamma))=\gamma((i-1)k^{2^{m-1}}+j)$ for all
$\gamma \in \{1, \dots, k^{2^{m-1}}\}^{\{1, \dots, k^{2^m}\}}$,
and set
$A_{m+1, (i-1)k^{2^{m-1}}+j}=A_{m,i}\cap s^{-1}_1\overline{U_{m, \omega_{i, j}(1)}}\cap \dots \cap s^{-1}_n\overline{U_{m, \omega_{i, j}(n)}}$.
Then properties (b),
(c), (e) and (f) hold for $m+1$. For each $1\le j\le k^{2^m}$ write
$A_{m+1, j}$ as the union of finitely many closed subsets each with
diameter no bigger than $2^{-(m+1)}$. Using
Lemma~\ref{L-decomposition E} we may replace $A_{m+1, j}$ by one of
these subsets. Consequently, property (d) is also satisfied for
$m+1$. This completes the induction procedure and hence the proof of
the theorem.
\end{proof}

\begin{corollary}\label{C-positive entropy to chaos}
If $h_\Sigma (X,G) > 0$ for some sofic approximation sequence $\Sigma$ then the action is Li-Yorke chaotic.
\end{corollary}

An action $G\curvearrowright X$ is said to be {\it distal} if $\inf_{s\in G}\rho(sx, sy)>0$ for all distinct $x, y\in X$.
We refer the reader to \cite{Auslander} for the basics of distal actions. Since distal actions have no Li-Yorke pairs,
from Corollary~\ref{C-positive entropy to chaos} we obtain the following consequence, which
extends the result of Parry that distal integer actions on compact metrizable spaces have zero entropy \cite{Parry}.
For amenable $G$ we write $h_\topol (X,G)$ for the classical topological entropy, which is equal to the
sofic entropy $h_\Sigma (X,G)$ for every $\Sigma$ \cite{KerLi10}.

\begin{corollary}\label{C-distal}
If the action $G\curvearrowright X$ is distal, then
$h_\Sigma (X,G)=0$  or $-\infty$. In particular, if $G$ is amenable and $G\curvearrowright X$ is distal, then $h_\topol (X,G)=0$.
\end{corollary}

We remark that every distal action has an invariant Borel probability measure \cite[page 125]{Auslander} \cite[page 496]{Vries}.
But we do not know whether one can conclude that $h_\Sigma (X,G)=0$ in Corollary~\ref{C-distal}.

\end{document}